\documentclass[a4paper,10pt]{article}
\usepackage[english]{babel}
\usepackage{amsthm,amsmath,amssymb,bbm}
\usepackage{stackengine}
\usepackage[numbers]{natbib}

\usepackage{tikz}
\usetikzlibrary{positioning, calc}

\newtheorem{thm}{Theorem}[section]

\newtheorem{ex}[thm]{Example}
\newtheorem{prop}[thm]{Proposition}
\newtheorem{defin}[thm]{Definition}
\newtheorem{rema}[thm]{Remark}
\newtheorem{corollary}[thm]{Corollary}
\usepackage{xcolor}

\def \n{\Vert}

\def\E{{\mathbb{E}}}
\def\H{{\mathsf{H}}}
\def\C{{\mathbb{C}}}

\def\P{{\mathbb{P}}}
\def\R{{\mathbb{R}}}
\def\N{{\mathbb{N}}}

\def\F{{\cal{F}}}

\newcommand{\norm}[1]{\left\lVert #1 \right\rVert}

\newcommand{\ran}{\operatorname{ran}}
\newcommand{\Ker}{\operatorname{ker}}

\newcommand{\Pncoef}{\mathcal{P}_n}

\newcommand{\pairHH}[2]{\left\langle #1,#2 \right\rangle_{\H,\H^*}}

\def\limt{\lim_{t\to\infty}}
\def\limt0{\lim_{t\to 0}}

\def\|{\,|\,}

\def\bn#1\en{\begin{align*}#1\end{align*}}
\def\bnn#1\enn{\begin{align}#1\end{align}}

\allowdisplaybreaks

\title{An Extremal Reconstruction Principle under Covariance Domination}

\author{Philip Kennerberg\footnote{
Email: pkennerberg@gmail.com}} 
\date{}

\begin{document}
	\maketitle

\begin{abstract}
We identify a structural extremal principle governing
residual \(L^2\)-norms over operator-ordered covariance
envelopes.
In contrast to the centered setting, where such quantities reduce to
trace expressions involving covariance operators, the non-centered
framework generates mixed terms that cannot be recovered from covariance
ordering alone.

We show that the worst-case squared residual \(L^2\)-norm over an
operator-ordered covariance envelope is attained at a canonical envelope
representative, possibly belonging only to the closure of the admissible
class. The resulting extremal identity holds uniformly over all
admissible reconstruction operators.

The result is obtained without convexity, compactness, or a global Hilbert space structure governing all components of the system.

As a consequence, the associated minimax reconstruction problem over
covariance envelopes reduces to evaluation at a canonical representative
under covariance domination.
\end{abstract}

\section{Introduction}
We study an operator-theoretic extremal principle for least-squares
reconstruction under covariance domination.
The basic setting consists of a target component and an observed component,
both generated from a common source through a fixed bounded operator.
Given a source $A$ and an associated additive baseline component $\xi^A$,
we consider pairs
\[
(Y^A,X^A)=\mathcal S(A+\xi^A),
\]
where
\[
\mathcal S:\mathbb H_o^d\to \mathbb H_o\times \H^*
\]
is a bounded linear map,
$\mathbb H_o$ is a separable Hilbert space,
and $\H^*$ denotes the dual of a separable Hilbert space $\H$.
The component $X^A\in\H^*$ is interpreted as the observed component,
while $Y^A\in\mathbb H_o$ is the target component to be reconstructed.


This framework should be viewed as a structural reconstruction model rather
than a formulation tied to a single perturbative interpretation. The same
covariance-envelope mechanism admits two complementary viewpoints.

In the first viewpoint, the source component \(A\) represents the uncertain
or varying component of the system, while the baseline component \(\xi^A\)
represents a fixed, typically random, background structure (for example,
noise). In the second viewpoint, the baseline component \(\xi^A\)
represents the intrinsic reference structure, while the source component
\(A\) models structured perturbative deviations relative to this
reference structure.

The present paper considers two corresponding realizations of this
framework: a wide-sense stationary/linear time-invariant (WSS/LTI)
realization associated with the first viewpoint, and an elliptic PDE
reconstruction realization associated with the second viewpoint.
Functional regression models, including robust functional regression and
functional anchor regression, provide further instances of the second
viewpoint. 

The elliptic PDE example provides a reconstruction setting in which
covariance domination on the perturbative component \(A\) induces
covariance domination for the associated reconstructed quantities.
The WSS/LTI example places the covariance-envelope framework in the
setting of robust Wiener filtering under spectral uncertainty, where
covariance domination becomes a pointwise Loewner-order constraint on
matrix-valued spectral densities. In contrast to classical minimax
formulations based on optimization over admissible spectral
classes~\cite{Poor1980}, the present framework yields an extremal
reduction in which the maximal residual cost over the covariance
envelope is attained at a canonical representative.

The framework also includes models without baseline components
(i.e.\ \(\xi^A=0\)), although the presence of baseline components allows for
substantially richer perturbative structures.
\noindent
To reconstruct $Y^A$ from $X^A$, we consider operators
\[
T\in\mathcal{HS}(\H^*,\mathbb H_o),
\]
and define the associated quadratic residual cost
\[
R_A(T)
=
\int_\Omega
\lVert Y^A-T(X^A)\rVert_{\mathbb H_o}^2\,d\mu .
\]
The admissible sources are assumed to satisfy a covariance domination constraint.
Given a family of admissible sources $\mathcal A$
and a reference source $A\in\overline{\mathcal A}$, we consider the associated covariance envelope
\[
C_{\mathcal A}(A)
=
\{A'\in\mathcal A:\Sigma_{A'}\preceq \Sigma_A\},
\]
where the ordering is understood in the Löwner sense.
The main result of the paper establishes the extremal identity
\[
\sup_{A'\in C_{\mathcal A}(A)}R_{A'}(T)
=
R_A(T)
\]
for every fixed admissible operator $T$.
Thus the maximal residual over the covariance envelope is attained at a canonical representative.
The result is nontrivial for several reasons.
The covariance envelopes are in general neither convex nor closed,
the extremal representative may exist only in the closure of the admissible class,
and the residual cost cannot be reduced to a monotone functional
of a single covariance operator.
For each admissible source $A$, the associated baseline component $\xi^A$
is assumed to satisfy
\[
\int_\Omega \xi^A\otimes\xi^A\,d\mu=\Sigma_\xi,
\qquad
\int_\Omega A\otimes\xi^A\,d\mu=0.
\]
Thus the baseline contribution has fixed second--order structure and is
orthogonal to the source component at the covariance level.

\medskip\noindent
\textbf{Extremal envelope principle (informal statement).}
Fix a reference source \(A\in\overline{\mathcal A}\) and let
\(C_{\mathcal A}(A)\) denote the associated covariance envelope.
For every admissible reconstruction operator
\[
T\in\mathcal{HS}(\H^*,\mathbb H_o),
\]
the maximal residual \(L^2\)-norm over the entire envelope is attained
by the canonical representative \(A\):
\[
\sup_{A'\in C_{\mathcal A}(A)}
R_{A'}(T)
=
R_A(T).
\]
The extremal identity therefore reduces the associated minimax
reconstruction problem over the covariance envelope to evaluation at the
canonical representative.
When optimization over admissible operators $T$ is subsequently
considered, the extremal identity collapses the envelope problem to a
single least-squares minimization problem.
In the WSS/LTI setting this leads to a frequency--domain characterization
of the associated reconstruction operator through matrix-valued spectral
densities and operator-valued normal equations.

\medskip\noindent


\medskip\noindent
\textbf{Relation to reconstruction, robust filtering, and operator theory.}
The resulting framework yields a canonical envelope reduction for
infinite-dimensional reconstruction problems under covariance
domination.

The WSS/LTI specialization admits a spectral-domain formulation under covariance
domination. In this setting, the covariance domination condition
\[
\Sigma_{A'}\preceq \Sigma_A
\]
induces a pointwise Loewner-order constraint on the associated
matrix-valued spectral densities. Classical robust filtering
formulations typically rely on minimax optimization over admissible
spectral classes~\cite{Poor1980,KassamPoor1985}, whereas the present
framework reduces the minimax reconstruction problem to evaluation at a
canonical envelope representative. From an operator-theoretic viewpoint,
the associated reconstruction problem is also related to variational
formulations in Hilbert spaces of the type developed by
Luenberger~\cite{Luenberger1969}. The WSS/LTI specialization therefore
provides a spectral realization of the general reconstruction framework
within the analysis of linear dynamical systems.

From a functional-analytic viewpoint, the present extremal mechanism is
distinct from classical spectral variational principles such as the
Rayleigh--Ritz and Courant--Fischer characterizations arising in
operator and matrix analysis
(see, e.g.,~\cite{Bhatia1997,Conway1990}). In contrast to spectral
variational principles based on subspace optimization and eigenvalue
characterizations, the present framework concerns extremal reconstruction
costs over operator-ordered covariance envelopes.

\medskip\noindent
\textbf{Organization of the paper.}
Section~2 introduces the covariance-envelope framework and establishes
the envelope extremal principle.
Section 3 studies the associated extremal minimization problem, derives the
operator-valued normal equations, and analyzes the resulting reconstruction
operators in abstract form and in elliptic and WSS/LTI realizations.

\subsection{Operator-theoretic framework}\label{sec:sfr}

We work on a fixed finite measure space $(\Omega,\F,\mu)$. 
Let $\H$ denote a real separable Hilbert space. 
We consider the vector space
\[
V_1 =\mathcal{HS}(\H^*,\mathbb{H}_o),
\]
where $\H^*$ is the topological dual of $\H$. Typical realizations include the following.
\begin{itemize}
\item \textbf{\(L^2\)-kernel realization.}
Take $\H=\mathbb H_o=L^2([t_1,t_2])$. Then $\H^*\cong L^2([t_1,t_2])$ and
\[
V_1=\mathcal{HS}(\H^*,\mathbb H_o)\cong L^2([t_1,t_2]^2),
\]
via the standard correspondence between Hilbert Schmidt operators and $L^2$ kernels.

\item \textbf{Abstract Hilbert Schmidt case (tensor form).}
Let $\H$ be any real separable Hilbert space and let $\mathbb H_o$ be any real separable Hilbert space.
Then
\[
V_1=\mathcal{HS}(\H^*,\mathbb H_o)\cong \mathbb H_o\widehat\otimes \H,
\]
canonically and isometrically. This covers, for example, $\H=\ell^2$ and $\mathbb H_o=\ell^2$,
in which case $V_1$ is the space of Hilbert Schmidt matrices.

\item \textbf{Sobolev duality for the source space.}
Fix $r>0$ and take $\H=H^{r}([t_1,t_2])$, so $\H^*=H^{-r}([t_1,t_2])$.
Then $V_1=\mathcal{HS}(H^{-r},\mathbb H_o)$.

\end{itemize}

\paragraph{Dual pairing.}  
For $h\in \H$ and $\ell\in \H^*$ we write $\pairHH{h}{\ell}:=\ell(h)$.  
Via the Riesz isometry $\mathcal R_\H:\H\to \H^*$, $(\mathcal R_\H v)(h)=\langle h,v\rangle_\H$, 
we identify $\H$ with $\mathcal R_\H(\H)\subset \H^*$.  
Thus, whenever $v\in \H$, 
\[
\pairHH{h}{v}:=\pairHH{h}{\mathcal R_\H v}=\langle h,v\rangle_\H .
\]

\paragraph{The inputs.}  
The primitive elements of the system are \emph{baseline components}
and \emph{source components}. Both take values in \(\mathbb H_o^d\).
We define
\[
\mathcal V
=
\Bigl\{
U:\Omega\to\mathbb H_o^d
\ \text{$\F$--measurable} :
\sum_{i=1}^d
\int_\Omega
\|U(i)\|_{\mathbb H_o}^2\,d\mu
<
\infty
\Bigr\}.
\]
Equipped with the inner product
\[
\langle U_1,U_2\rangle_{\mathcal V}
=
\sum_{i=1}^d
\int_\Omega
\langle U_1(i),U_2(i)\rangle_{\mathbb H_o}\,d\mu,
\]
the space \(\mathcal V\) becomes a real separable Hilbert space.
All integrals of \(\mathbb H_o^d\)-valued elements are understood in the
Bochner sense.
Let $\Sigma_\xi$ be a fixed positive semidefinite operator on $\mathbb H_o^d$.
For each admissible source $A$ we assume the existence of a corresponding
baseline component $\xi^A \in \mathcal V$ such that
\[
\int_\Omega \xi^A \otimes \xi^A \, d\mu = \Sigma_\xi,
\qquad
\int_\Omega A \otimes \xi^A \, d\mu = 0
\qquad\text{in }\mathcal L(\mathbb H_o^d),
\]
where \(\mathcal L(\mathbb H_o^d)\) denotes the space of bounded linear
operators on \(\mathbb H_o^d\).

\paragraph{System operator.}
Let
\[
\mathcal S=(\mathcal S_Y,\mathcal S_X):
\mathbb H_o^d\to \mathbb H_o\times \H^*
\]
be a bounded linear map, where
\[
\mathcal S_Y:\mathbb H_o^d\to \mathbb H_o,
\qquad
\mathcal S_X:\mathbb H_o^d\to \H^*.
\]
For each admissible source $A$ we consider the associated observation-target system
\[
(Y^A,X^A)=\mathcal S(A+\xi^A).
\]
This structure is natural in settings where perturbations act prior to
the system dynamics or observation mechanism, so that covariance
domination is transported through the induced target and observation
components. For example, as we shall see, forcing terms in PDE models and latent
input perturbations in WSS/LTI systems enter prior to the action of the
associated system operator.

\medskip
\noindent\textbf{Structural features of the envelope problem.}
The extremal problem considered here differs in several structural respects from classical spectral minimax and Rayleigh–Ritz type settings.
First, the admissible covariance envelopes
\[
C_{\mathcal A}(A)
=
\{A'\in\mathcal A:\Sigma_{A'}\preceq\Sigma_A\}
\]
are generally neither convex nor closed and need not contain an extremal
element.
As a consequence, compactness and extremizer-based arguments are not
available.

Second, the covariance structure does not reduce to a single Hilbertian
geometry.
While the residual cost is naturally defined on
$L^2(\Omega;\mathbb H_o)$, the auxiliary component takes values in the
dual space $\H^*$ and interacts with reconstruction operators through
duality pairings.
Consequently, the full block structure does not reduce to a single Hilbertian quadratic form, and the analysis cannot be formulated purely through orthogonality or spectral decomposition arguments.
Third, the sources are not assumed to be centered.
As a consequence, the residual cost cannot in general be reduced to a
trace functional of the covariance operator.
Even in a Hilbertian setting, mixed source--baseline contributions prevent
identities of the form
\[
\int_\Omega
\langle A'(\omega),MA'(\omega)\rangle\,d\mu(\omega)
=
\operatorname{tr}(\Sigma_{A'}M),
\]
so covariance domination alone does not determine extremality.

	\section{Envelope extremal principle}
	\label{sec:worstrisk}

For $A'\in\mathcal V$ and $1\le i,j\le d$, define the second-order operators
\[
\Sigma_{ij}
:=
\int_\Omega A'(i)\otimes A'(j)\,d\mu
\;:\;
\mathbb H_o\to\mathbb H_o,
\]
where $(x\otimes y)u := \langle u,y\rangle_{\mathbb{H}_o}\,x$.
Collecting the blocks we define the operator
\begin{equation}\label{Kernel}
\Sigma_{A'}
:=
\begin{bmatrix}
\Sigma_{11} & \cdots & \Sigma_{1, d }\\
\vdots      & \ddots & \vdots\\
\Sigma_{ d ,1} & \cdots & \Sigma_{ d , d }
\end{bmatrix}
\;:\; \mathbb{H}_o^{d}\to\mathbb{H}_o^{d}.
\end{equation}
Given $A\in\mathcal{V}$ and $T\in \mathcal{HS}(\H^*,\mathbb{H}_o)$, we define the associated \emph{cost functional} (denoted $R_A(T)$ throughout):
\begin{align*}
    R_A(T)
    := \int_\Omega  \left\lVert Y^A-T\left(X^{A} \right)  \right\rVert_{\mathbb{H}_o}^2 d\mu 
\end{align*}
By linearity and boundedness of the components of $\mathcal S$, we have 
$$\left\lVert  T(X^A) \right\rVert_{\mathbb{H}_o} \le   \left\lVert T\right\rVert\left\lVert \mathcal{S}_{X}(A+\xi^A) \right\rVert_{\H^*},$$
and since $ \lVert T \rVert \le \left\lVert T\right\rVert_{\mathcal{HS}(\H^*,\mathbb{H}_o)} $ this ensures that $R_A(T)$ is finite and well-defined.


We now introduce the central covariance envelope associated with a
reference source.
For a given source $A$ and a class of admissible sources $\mathcal{A}$, the covariance envelope consists of all 
sources whose covariance structure is dominated, in the Loewner sense, 
by $A$.

\begin{defin}\label{ShiftDef} {\bf Covariance envelope.}
Let $A\in\mathcal{V}$ be a source and $\mathcal{A}\subseteq\mathcal{V}$ a set of sources. Define
	\begin{align}\label{ShiftSet}
		C_{\mathcal{A}}(A)=\Big\{A'\in \mathcal{A}: \langle \mathbf g,\Sigma_{A'}\mathbf g\rangle_{\mathbb{H}_o^{d}} \le \langle \mathbf g,\Sigma_{A}\mathbf g\rangle_{\mathbb{H}_o^{d}} \quad \forall g\in \mathbb{H}_o^{d}
		\Big\}.
	\end{align}
\end{defin}

This definition is the natural multivariate analogue of Mercer’s condition for covariance kernels. 
\begin{rema}
The condition $A' \in C_{\mathcal{A}}(A)$ is equivalently expressed as the operator inequality 
\[
\Sigma_{A'} \preceq \Sigma_{A},
\]
where $\preceq$ denotes the order induced by quadratic forms.
\end{rema}
\noindent
This definition also admits the following equivalent formulation, which requires verification only on a dense subset. Let $\mathcal{G}\subseteq \mathbb{H}_o$ be such that $\overline{\mathcal{G}}=\mathbb{H}_o$. Then

\begin{prop}\label{Gprop}
	\begin{align*}
		C_{\mathcal{A}}(A)=\Big\{A'\in \mathcal{A}: \langle \mathbf g,\Sigma_{A'}\mathbf g\rangle_{\mathbb{H}_o^{d}} \le \langle \mathbf g,\Sigma_{A}\mathbf g\rangle_{\mathbb{H}_o^{d}},  \quad\forall \mathbf g\in \mathcal G^d
		\Big\}.
	\end{align*}
\end{prop}
\noindent Another elementary topological property of the envelope is the following.

\begin{prop}\label{closed}
	$C_{\mathcal{A}}(A)$ is closed in $\mathcal{V}$ whenever $\mathcal{A}$ is closed in $\mathcal{V}$.
\end{prop}
\noindent We next present three special cases where we characterize $C_{\mathcal{A}}(A)$ explicitly.

\begin{ex}
The covariance envelope condition reduces in finite dimensions to a matrix
inequality for the coefficient second-order operators.
Let $\{\phi_1,\ldots,\phi_n\}$ be orthonormal and let
$\mathcal{A}=\mathsf{span}\{\phi_1,\ldots,\phi_n\}$.
If $A(i)=\sum_{k=1}^n a_{i,k}\phi_k$ with $a_{i,k}\in L^2(\Omega)$, then
$C_{\mathcal{A}}(A)$ consists of sources of the form
$A'(i)=\sum_{k=1}^n a'_{i,k}\phi_k$, $1\le i\le  d $, with $a'_{i,k}\in L^2(\Omega)$,
whenever
\[
\int_\Omega \mathbf{a'}^T\mathbf{a'}\,d\mu
\ \preceq\
\int_\Omega \mathbf{a}^T\mathbf{a}\,d\mu ,
\qquad
\mathbf{a}=(a_{1,1},\ldots,a_{ d ,n}),\ \
\mathbf{a'}=(a'_{1,1},\ldots,a'_{ d ,n}).
\]
\end{ex}

\noindent\textbf{Stationary subclass (restriction to $[t_1,t_2]$).}
In this subsection we specialize to the probabilistic case $\mu=\P$ and write
$\E[\cdot]=\int_\Omega (\cdot)\,d\P$. We take $\mathbb H_o=L^2([t_1,t_2])$ and consider
$ d $-variate sources $A=(A(1),\ldots,A(d))$ with $A(i)\in L^2(\Omega;\mathbb H_o)$.

\begin{prop}[Wide-sense stationary restrictions]\label{WSS}
Assume that each \(A\in\mathcal A\subset\mathcal V\) arises as the
restriction to \([t_1,t_2]\) of a mean-zero wide-sense stationary
\(d\)-variate process on \(\mathbb R\) with covariance function
\(K_A:\mathbb R\to\mathbb R^{d\times d}\), i.e.
\[
\mathbb E\!\left[A(s)A(t)^\top\right]
=
K_A(s-t)
\qquad\text{for all }s,t\in\mathbb R,
\]
and suppose that each entry of \(K_A\) belongs to \(L^1(\mathbb R)\),
so that the matrix-valued Fourier transform
\(\widehat K_A\) is well-defined and bounded.
Then, for $A,A'\in\mathcal A$,
\[
A'\in C_{\mathcal A}(A)
\quad\Longleftrightarrow\quad
\widehat K_A(\omega)-\widehat K_{A'}(\omega)\ \text{is positive semidefinite for Lebesgue-a.e. }\omega\in\R,
\]
where $\widehat K_A(\omega)$ denotes the matrix Fourier transform of $K_A$ (taken entrywise).
Equivalently,
\[
C_{\mathcal A}(A)=\Bigl\{A'\in\mathcal A:\ \widehat K_A(\omega)-\widehat K_{A'}(\omega)\succeq 0
\ \text{for Lebesgue-a.e. }\omega\in\R\Bigr\}.
\]
\end{prop}
Thus, in the WSS/LTI setting the covariance envelope induces a pointwise
spectral domination structure in the frequency domain.
This places the extremal principle within the broader context of
robust Wiener filtering and spectral uncertainty theory.
We now state the main structural result, the \emph{Envelope extremal principle}, which shows that the supremum of the cost functional over the covariance envelope is attained at the reference source itself.
\noindent
\begin{thm}\label{WR1} {\bf Envelope extremal principle.}
If $A \in \mathcal{V}$ and $A\in\bar{\mathcal{A}}$ then
\[
\sup_{A'\in C_{\mathcal{A}}(A)} R_{A'}(T)\;=\;R_{A}(T),
\qquad\text{for every } T \in V_1.
\]
\end{thm}
Without such an extremal reduction, the corresponding
minimax problem would require optimization over an entire covariance envelope,
which is typically nonclosed, nonconvex, and infinite-dimensional.
The extremal value depends only on the closure of the admissible class:
replacing \(\mathcal A\) by \(\overline{\mathcal A}\) leaves the supremum
unchanged.
\begin{corollary}
If $A \in \mathcal{V}$ and $A\in \bar{\mathcal{A}}$, then
\[
 \sup_{A'\in C_{\mathcal{A}}(A)} R_{A'}(T)
 \;=\;
 \sup_{A'\in C_{\bar{\mathcal{A}}}(A)} R_{A'}(T).
\]
\end{corollary}
The corollary ensures that working with explicit descriptions of $C_{\mathcal{A}}(A)$ 
is sufficient, even though the closure $C_{\bar{\mathcal{A}}}(A)$ 
may be analytically intractable. 

\subsection{Overview of the proof of Theorem~\ref{WR1}}
The first step is to expand the target $Y^A$ in an orthonormal basis 
$\{\phi_k\}_{k}$ of $\mathbb{H}_o$, and to use the fact that $\mathcal{HS}(\H^*,\mathbb{H}_o)\cong \mathbb{H}_o \widehat{\otimes}\H$ to identify an element $\beta_T\in \mathbb{H}_o \widehat{\otimes}\H$ through the canonical isometric isomorphism. We can then represent 
$\beta_T$ in the product basis $\{\phi_k\otimes \psi_l\}_{k,l}$, 
where $\{\psi_l\}_{l}$ is an orthonormal basis of $\H$. 
Next, the dual pairing between $T$ and $X^A$ is expanded in the 
$\{\phi_k\}$–basis, so that the entire residual is expressed consistently 
in terms of these coordinates. 
This yields an explicit coefficient representation of the cost functional.
\\
The next step is to approximate the infinite–dimensional functional system 
by a finite–dimensional truncation. Concretely, we introduce matrices $B^n$ 
encoding the finite–dimensional action of $\mathcal{S}$ relative to the chosen ON–basis. 
Since S is not assumed compact, the truncations are handled through pointwise approximation rather than operator-norm convergence, and the resulting approximation error must therefore be controlled in the limiting procedure.
\\
Next, the finite–dimensional approximation is substituted into the expansion 
of the cost functional, and the pointwise error is controlled by a detailed 
estimate. This reduction yields a finite–dimensional quadratic form expressed 
in terms of the Fourier coefficients of the input sources together with the 
corresponding baseline component terms.

In the resulting quadratic form we show that the mixed source--baseline
terms vanish, while the pure baseline contribution is invariant.
This structure yields continuity of the cost functional with respect to
the source component. Owing to the preceding expansions, the dependence on the source and the
baseline component separates.

The orthogonality of the mixed terms together with the fixed second-order
structure of the baseline component implies that differences of cost
functionals depend only on source-dependent contributions.
This removes all baseline dependence from the envelope comparison and
reduces the extremal problem to quadratic forms in the source
coefficients.

We now proceed to the envelope maximization problem and
decompose the final step into several substeps.
\begin{itemize}
\item[\textbf{(a)}] \textit{Finite approximants inside the envelope.}
Fix $\Delta>0$. Using the continuity property established in the previous step, pick $A_\Delta\in C_{\mathcal{A}}(A)$ with
$\mathrm{dist}(A_\Delta,\,A)\le\Delta$.
Construct finite sets $C_m\subset C_{\mathcal{A}}(A)$ such that
(i) $\max_{A''\in C_m} R_{A''}(T)\to \sup_{A'\in C_{\mathcal{A}}(A)} R_{A'}(T)$ as $m\to\infty$, and
(ii) $A_\Delta\in C_m$ for all $m$. Since each $C_m$ is finite, the convergences used below are uniform over $C_m$.
The envelope inequality, although formulated in infinite dimension,
passes to finite-dimensional compressions of the covariance kernels,
thereby allowing the approximation machinery to apply.

\item[\textbf{(b)}] \textit{Finite-dimensional compression of kernels.}
Let $P_N:\mathbb H_o\to\mathbb H_o$ be the orthogonal projection onto
$\mathsf{span}\{\phi_1,\ldots,\phi_N\}$ and extend it componentwise to
$\mathbb H_o^{d}$ by
\[
\mathbf P_N := \mathrm{diag}(P_N,\ldots,P_N)\,:\,\mathbb H_o^{d}\to \mathbb H_o^{d}.
\]
For each $A''\in\overline{\mathcal A}$ define the compressed kernel by
$
\Sigma^{(N)}_{A''}:=\mathbf P_N\,\Sigma_{A''}\,\mathbf P_N$.
Then for every $\mathbf g\in (\mathsf{span}\{\phi_1,\ldots,\phi_N\})^{d}$ we have
$
\langle \mathbf g,\Sigma_{A''}\mathbf g\rangle_{\mathbb H_o^{d}}
=
\langle \mathbf g,\Sigma^{(N)}_{A''}\mathbf g\rangle_{\mathbb H_o^{d}},
$
since $\mathbf P_N\mathbf g=\mathbf g$. Hence the compression does not alter the
quadratic form on the finite-dimensional test space.

\item[\textbf{(c)}] \textit{Transfer of the envelope inequality to the truncated level.}
By the envelope definition,
$\langle \Sigma_{A''}\mathbf g,\mathbf g\rangle \le \,\langle \Sigma_A \mathbf g,\mathbf g\rangle$ for all $\mathbf g$.
For $\mathbf g\in\mathsf{span}\{\phi_1,\ldots,\phi_N\}$ the identities in (b) give
\[
\langle \Sigma^{(N)}_{A''}\mathbf g,\mathbf g\rangle_{\mathbb{H}_o^{d}}
=\langle \Sigma_{A''}\mathbf g,\mathbf g\rangle_{\mathbb{H}_o^{d}}
\le \,\langle \Sigma_A \mathbf g,\mathbf g\rangle_{\mathbb{H}_o^{d}}
= \,\langle \Sigma_A^{(N)} \mathbf g,\mathbf g\rangle_{\mathbb{H}_o^{d}},
\]
i.e. the envelope condition \emph{passes} to the compressed kernels on the finite subspace.

\item[\textbf{(d)}] \textit{Upper bound via truncation and limits.}
Applying (c) with test functions $\mathbf g$ determined by the finite
expansion of the truncated cost functional,
we obtain
$
\max_{A''\in C_m} R_{A''}^{(N)}(T) \;\le\; \,R_A^{(N)}(T).
$
Letting $N\to\infty$ and using the previously established convergence estimates,
together with uniformity over the finite set $C_m$, we then obtain
$
\max_{A''\in C_m} R_{A''}(T) \;\le\; R_{A}(T)+\Delta.
$
\item[\textbf{(e)}] \textit{Lower bound via a fixed near–candidate and passage $\Delta\to 0$.}
Since $A_\Delta\in C_m$ for all $m$,
$
\max_{A''\in C_m} R_{A''}(T) \;\ge\; R_{A_\Delta}(T).
$
Letting $m\to\infty$ gives
$\sup_{A'\in C_{\mathcal{A}}(A)} R_{A'}(T)\ge R_{A_\Delta}(T)$.
Finally, by source continuity (established earlier) and $\mathrm{dist}(A_\Delta,A)\le\Delta$,
$R_{A_\Delta}(T)\to R_{A}(T)$ as $\Delta\to 0$.
Combining with (d) yields
$
\sup_{A'\in C_{\mathcal{A}}(A)} R_{A'}(T) \;=\; R_{A}(T).
$

\end{itemize}

	\section{Envelope extremal minimization}
	\label{sec:minimizer}

We study the variational problem of minimizing the \emph{extremal (envelope) cost}
\[
\arg\min_{T\in \mathcal{HS}(\H^*,\mathbb{H}_o)}\ \sup_{A'\in C_{\mathcal A}(A)} R_{A'}(T),
\]
where $\arg\min$ denotes the (possibly empty) set of minimizers. Our goal is to characterize this set
and to give conditions for non-emptiness and uniqueness. We refer to this problem as
\emph{envelope extremal minimization}.
The minimization problem identifies reconstruction operators that are
optimal over the entire covariance envelope.
The resulting first-order optimality condition takes the form of an
operator-valued normal equation.

Fix $A\in\mathcal{V}$ and let $\mathcal{H}=L^2(\Omega;\mathbb{H}_o)$. Define the bounded linear operator
\[
\Gamma:\ \mathcal{HS}(\H^*,\mathbb{H}_o)\longrightarrow \mathcal{H},\qquad
\Gamma T :=\ T(X^A)  \in \mathbb{H}_o,
\]
interpreting the right-hand side in $\mathcal{H}$.
With this notation,
\[
R_{A}(T)\ =\ \norm{\Gamma T - Y^{A}}_{\mathcal{H}}^{2},\qquad T\in \mathcal{HS}(\H^*,\mathbb{H}_o).
\]
Write $J_{\mathcal{H}}:\mathcal{H}\to \mathcal{H}^{\ast}$ for the Riesz isomorphism, $J_{\mathcal{H}}(u)(v)=\langle u,v \rangle_{\mathcal{H}}$, and let $\Gamma^{\ast}:\mathcal{H}^{\ast}\to \mathcal{HS}(\H^*,\mathbb{H}_o)^{\ast}$ be the (Banach) adjoint, $\Gamma^{\ast}\ell:=\ell\circ\Gamma$. Set $V_1=\mathcal{HS}(\H^*,\mathbb{H}_o)$
\[
C_{XX}\ :=\ \Gamma^{\ast}J_{\mathcal{H}}\Gamma\ \in \mathcal{L}(V_1,(V_1)^{\ast}),
\qquad
C_{XY}\ :=\ \Gamma^{\ast}J_{\mathcal{H}}Y^A\ \in (V_1)^{\ast}.
\]

\begin{thm}[Extremal cost minimizer]\label{minim}
\begin{enumerate}
\item[(i)] \emph{First-order optimality / normal equation.}
There exists at least one solution $T^*\in V_1$,
\begin{equation}
T^*=\arg\min_{T\in V_1}\sup_{A'\in C_{\mathcal{A}}(A)} R_{A'}(T),
\end{equation}
if and only if
\begin{equation}
C_{XX}T^{\ast}=C_{XY}\quad\text{in }(V_1)^{\ast}.
\end{equation}
Equivalently, a minimizer exists iff $C_{XY}\in\ran(C_{XX})\subset (V_1)^{\ast}$.

\item[(ii)] \emph{Structure/uniqueness.}
If $T_0$ is any solution of $C_{XX}T=C_{XY}$, then the set of all minimizers is the affine space
\begin{equation}
T_0+\Ker(C_{XX}).
\end{equation}
Uniqueness holds precisely when $\Ker(\Gamma)=\{0\}$.
\end{enumerate}
\end{thm}
\noindent Operator-valued normal equations of this type are classical in linear
estimation and filtering theory; see, for example,
Kailath, Sayed, and Hassibi~\cite{KailathSayedHassibi2000}.
The novelty in the present setting lies in the covariance-envelope
extremal structure underlying the minimization problem.
\medskip
\noindent If we now write the cost functional as
\[
R_{A}(T)\;=\;C_{A}\ -\ 2\,\mathcal L_{A}(T)\ +\ \mathcal Q_{A}(T,T),
\qquad T\in V_1,
\]
where $C_{A}=\int_\Omega\lVert Y^{A}\rVert_{\mathbb{H}_o}^2 d\mu$,
$\ \mathcal L_{A}:V_1\to\R$ is the continuous linear functional
\[
\mathcal L_{A}(T)\;:=\; \!\int_\Omega \langle
Y^{A},T(X^{A})\rangle_{\mathbb{H}_o}d\mu ,
\]
and $\ \mathcal Q_{A}:V_1\times V_1\to\R$ is the continuous, symmetric, positive semidefinite bilinear form
\[
\mathcal Q_{A}(T_1,T_2)\;:=\int_\Omega 
\langle T_1(X^A),T_2(X^A)\rangle_{\mathbb{H}_o}d\mu .
\]
Let $\{\phi_k\}_{k\ge1}$ be an ON–basis of $\mathbb{H}_o$ and $\{\psi_\ell\}_{\ell\ge1}$ an ON–basis of $\H$.
Define the (infinite) coefficient vectors $v=\big(v_{k,\ell}\big)_{(k,\ell)}\in\ell^2$,
$$v_{k,\ell}=\left\langle T(\psi_l),\phi_k\right\rangle_{\mathbb{H}_o}$$
and
\[
b_{k,\ell}\ :=\int_\Omega \left\langle Y^{A}\,
,\phi_k\right\rangle_{\mathbb{H}_o} \langle \psi_\ell, X^{A}\rangle_{\H,\H^*}d\mu ,
\]
together with the positive semidefinite operator $\Sigma:\ell^2\to\ell^2$ with entries
\[
\Sigma^{(k,\ell)\,(k',\ell')}\ :=\delta_{k,k'}\int_\Omega
\langle \psi_\ell, X^{A}\rangle_{\H,\H^*}
\langle \psi_{\ell'}, X^{A}\rangle_{\H,\H^*} d\mu ,
\]
here the infinite matrix $\Sigma$ acts on vectors in $\ell^2$ through left-multiplication. 
\\
\\
The following corollary provides an explicit coordinate representation
of the minimizer in a fixed orthonormal basis.
Under coercivity, the coefficient sequence belongs to $\ell^2$ and the
corresponding expansion converges in $V_1$.


\begin{corollary}[Coordinate form and summability for the minimizer]\label{cor:coord-min}
\noindent{(i) \bf Coercive case.}
If there exists $c>0$ such that
$$\mathcal Q_{A}(T,T) \ge c\lVert T\rVert_{V_1}^2\qquad\forall T\in V_1$$
(coercivity of the $XX$–block), then there exists a unique solution $ T^*\in V_1$,
$$T^*=\arg\min_{T\in V_1}\sup_{A'\in C_{\mathcal{A}}(A)} R_{A'}(T),$$
characterized by the normal equation
\[
\mathcal Q_{A}(T_A^\star,\eta)\ =\ \mathcal L_{A}(\eta)\qquad \forall\,\eta\in V_1,
\]
and with the bound $\lVert T_{A}^\star\rVert_{V_1}\le c^{-1}\lVert\mathcal L_{A}\rVert_{(V_1)^*}$.

\smallskip

\noindent{(ii) \bf Coordinate representation.}
Under the coercivity assumption above (which is equivalent to $\Sigma\succeq c\,I$ on $\ell^2$),
the unique minimizer has coefficient vector $v^\star=\Sigma^{-1}b\in\ell^2$, hence
\[
\sum_{k,\ell\ge1} \big|v_{k,\ell}\big|^2\ <\ \infty.
\]
Moreover, for each $x\in\H^*$,
\[
T_A(x)=\sum_{k,\ell\ge1} v_{k,\ell}\,\langle \psi_\ell,x\rangle_{\H,\H^*}\,\phi_k,
\]
where the series converges in $\mathbb{H}_o$.

\smallskip
\noindent{(iii) \bf Degenerate case.}
If coercivity fails but $\mathcal L_{A}$ belongs to the closure of the range of the operator induced by $\mathcal Q_{A}$,
then minimizers exist and are characterized by the normal equation above; among them, the minimal
$\lVert\cdot\rVert_{V_1}$–norm solution corresponds to the $\ell^2$–minimal solution
$v^\star=\Sigma^\dagger b$ (Moore–Penrose) and still satisfies $\ v^\star\in\ell^2$.
\end{corollary}

\noindent


\begin{ex}[LTI system]\label{ex:LTI}
Let $(\Omega,\mathcal F,\mu)$ be a finite measure space and assume
$
\mathbb H_o=L^2([t_1,t_2]).
$
Let $\H$ be a real separable Hilbert space.

\medskip\noindent
\textbf{LTI system operator (fixed).}
Fix impulse responses $h,\phi\in L^1(\mathbb R)$ and define convolution on $\mathbb R$ by
\[
(h*u)(t):=\int_{\mathbb R} h(t-s)u(s)\,ds,
\qquad
(\phi*u)(t):=\int_{\mathbb R} \phi(t-s)u(s)\,ds.
\]
Let $\Psi:\mathbb H_o\to \H^*$ be a fixed bounded linear map.
Define the system operator
\[
\mathcal S:\ \mathbb H_o^{2}\longrightarrow \mathbb H_o\times \H^*,
\qquad
U=(U_1,U_2)\longmapsto \big(\mathcal S_Y U,\ \mathcal S_X U\big),
\]
by
\[
\mathcal S_Y U := (h*U_1)\big|_{[t_1,t_2]}\in \mathbb H_o,
\qquad
\mathcal S_X U := \Psi\!\left[(\phi*U_2)\big|_{[t_1,t_2]}\right]\in \H^*.
\]
Then
\[
\lVert \mathcal S_Y\rVert \le \lVert h\rVert_{L^1(\mathbb R)},
\qquad
\lVert \mathcal S_X\rVert \le \lVert \Psi\rVert\,\lVert \phi\rVert_{L^1(\mathbb R)}.
\]
Hence
\[
\mathcal S=(\mathcal S_Y,\mathcal S_X):
\mathbb H_o^{2}\longrightarrow \mathbb H_o\times \H^*
\]
is a bounded linear map. The operator $\mathcal S_X$ determines the observed process,
while $\mathcal S_Y$ determines the target process to be reconstructed.

\medskip\noindent
\textbf{Sources and observations.}
Let $\mathcal A\subset\mathcal V$ be a class of admissible sources.
For each $A\in\mathcal A$ let $\xi^A\in\mathcal V$ be a baseline component such that
\[
\int_\Omega \xi^A\otimes\xi^A\,d\mu=\Sigma_\xi,
\qquad
\int_\Omega A\otimes\xi^A\,d\mu=0.
\]
In the present LTI/WSS setting, the baseline component $\xi^A$
can be used to model additive noise or background disturbances with
fixed second--order structure.
Define the observed pair
\[
(Y^A,X^A):=\mathcal{S}(A+\xi^A),
\qquad\text{i.e.}\qquad
Y^A=\mathcal S_Y(A+\xi^A),\quad
X^A=\mathcal S_X(A+\xi^A).
\]
Thus $Y^A$ is $\mathbb H_o$–valued and $X^A$ is $\H^*$–valued.

\medskip\noindent
\textbf{Quadratic cost functional.}
For $T\in\mathcal{HS}(\H^*,\mathbb H_o)$ define
\[
R_A(T):=\int_\Omega
\big\lVert Y^A - T(X^A)\big\rVert_{\mathbb H_o}^2\,d\mu .
\]

\medskip\noindent
\textbf{Covariance envelope.}
For $A\in\overline{\mathcal A}$ define
\[
C_{\mathcal A}(A):=
\Big\{
A'\in\mathcal A:
\Sigma_{A'}\preceq \Sigma_A
\ \text{on }\mathbb H_o^2
\Big\},
\]
where
$\Sigma_A:=\int_\Omega A\otimes A\,d\mu$
is the source covariance operator on $\mathbb H_o^{2}$.


\medskip\noindent
\textbf{Extremal principle.}
For every fixed $T\in\mathcal{HS}(\H^*,\mathbb H_o)$,
$
\sup_{A'\in C_{\mathcal A}(A)} R_{A'}(T)=R_A(T),
$
with $\mathcal{S}$ fixed as above.


\paragraph{Remark (Stationarity and robust spectral uncertainty).}
The extremal principle itself does not require stationarity.
Wide--sense stationarity is introduced only to obtain a diagonal
frequency--domain representation of the covariance operators
and of the associated minimizer.

In the WSS/LTI setting the pair $(Y^A,X^A)$ may be interpreted as a
target--observation system generated by the fixed operator pair
$(\mathcal S_Y,\mathcal S_X)$, where $X^A$ represents the observed process
and the operator $T$ acts as a reconstruction filter for $Y^A$.
The covariance envelope then induces a pointwise Loewner-order constraint
on the associated matrix-valued spectral densities,
\[
\widehat K_A(\omega)-\widehat K_{A'}(\omega)\succeq0
\qquad\text{for a.e.\ }\omega\in\mathbb R.
\]
This places the present framework in the broader context of Wiener
filtering~\cite{Wiener1949} and spectral uncertainty theory. Classical
robust formulations, such as those studied by Poor~\cite{Poor1980},
typically rely on minimax or saddle-point arguments over uncertainty
classes of admissible spectra. In contrast, the present framework yields
a direct extremal reduction to a canonical representative through
covariance domination and operator ordering.

\medskip\noindent
\textbf{WSS/LTI specialization.}
Assume now that $\mu=\P$ and that each $A\in\mathcal A$ admits a wide–sense stationary
extension to $\mathbb R$ with matrix–valued covariance kernel
$K_A(\tau)\in\mathbb R^{2\times 2}$.
Then $\Sigma_A$ is induced by convolution with $K_A(\cdot)$, and writing
$\widehat K_A(\omega)$ for its Fourier transform,
\[
A'\in C_{\mathcal A}(A)
\quad\Longleftrightarrow\quad
\widehat K_A(\omega)-\widehat K_{A'}(\omega)
\ \text{is positive semidefinite for a.e.\ }\omega,
\]
cf.\ Proposition~\ref{WSS}.

\medskip\noindent
\textbf{Frequency–domain covariance blocks (representation).}
In the WSS/LTI setting, the covariance structure of $(Y^A,X^A)$
diagonalizes under the Fourier transform.
Denoting by
\[
\Sigma_{(Y^A,X^A)}=
\begin{pmatrix}
\Sigma_A^{YY} & \Sigma_A^{YX}\\
\Sigma_A^{XY} & \Sigma_A^{XX}
\end{pmatrix},
\]
the corresponding spectral densities satisfy, for a.e.\ $\omega$,
\begin{align*}
  \widehat K_A^{YY}(\omega)
  &= |H(\omega)|^2\, [\widehat K_A(\omega)]_{11},\\
  \widehat K_A^{YX}(\omega)
  &= H(\omega)\,\overline{\Phi(\omega)}\,[\widehat K_A(\omega)]_{12},\\
  \widehat K_A^{XX}(\omega)
  &= |\Phi(\omega)|^2\,[\widehat K_A(\omega)]_{22},
\end{align*}
where $H=\widehat h$ and $\Phi=\widehat\phi$.
The baseline component contributes only through a fixed additive
positive semidefinite spectral term determined by $\Sigma_\xi$.

\medskip\noindent
\textbf{Frequency--domain characterization of the minimizer.}
The minimum--norm minimizer $T^*$ satisfies
$
K_A^{XX}T^*=K_A^{XY}.
$
In the WSS/LTI setting this diagonalizes under the Fourier transform into the
pointwise relation
\[
\widehat K_A^{XX}(\omega)\,\tau(\omega)
=
\widehat K_A^{XY}(\omega),
\qquad\text{for a.e.\ }\omega,
\]
where $\tau(\omega)$ denotes the frequency-domain transfer function
associated with the minimizer $T^*$.
Equivalently,
\[
\tau(\omega)
=
\big(\widehat K_A^{XX}(\omega)\big)^\dagger
\widehat K_A^{XY}(\omega).
\]

\end{ex}

\begin{ex}[Elliptic reconstruction under second-order uncertainty]
Let \(D\subset \mathbb R^m\) be a bounded Lipschitz domain and let
\[
a:H^1_0(D)\times H^1_0(D)\to\mathbb R
\]
be a bounded symmetric coercive bilinear form. By the Lax--Milgram theorem,
there exists a bounded solution operator
\[
G:H^{-1}(D)\to H^1_0(D)
\]
such that, for every \(f\in H^{-1}(D)\), the function \(u=Gf\) is the unique
weak solution of
\[
a(u,v)=\langle f,v\rangle_{H^{-1},H^1_0}
\qquad
\text{for all }v\in H^1_0(D).
\]

Let \(E\subset H^{-1}(D)\) be a separable Hilbert space continuously embedded
into \(H^{-1}(D)\). For each admissible perturbative forcing field
$
A\in L^2(\Omega;E),
$
let
$
\xi^A\in L^2(\Omega;E)
$
be a possibly \(A\)-dependent baseline forcing component with fixed second-order structure
\[
\int_\Omega \xi^A\otimes\xi^A\,d\mu=\Sigma_\xi,
\qquad
\Sigma_\xi\in\mathcal L(E),
\]
and satisfying
\[
\int_\Omega A\otimes\xi^A\,d\mu=0
\qquad
\text{in }\mathcal L(E).
\]

Define the second-order operator associated with the perturbative forcing
component by
\[
\Sigma_A
:=
\int_\Omega A(\omega)\otimes A(\omega)\,d\mu(\omega)
\in\mathcal L(E).
\]

The present reconstruction setting admits a perturbative interpretation
complementary to the WSS/LTI realization. Here the baseline forcing
component \(\xi^A\) represents the intrinsic second-order reference
structure, while the source component \(A\) models structured perturbative
deviations relative to this reference structure.

Define the corresponding random elliptic state by
\[
u^A(\omega):=G(A(\omega)+\xi^A(\omega)).
\]

Fix bounded linear operators
\[
S_Y:H^1_0(D)\to H_o,
\qquad
S_X:H^1_0(D)\to H^*,
\]
and define
\[
Y^A:=S_Yu^A,
\qquad
X^A:=S_Xu^A.
\]

Equivalently,
\[
(Y^A,X^A)(\omega)=\mathcal S(A+\xi^A),
\]
where
\[
\mathcal S:=(S_YG,S_XG):E\to H_o\times H^*
\]
is a bounded linear system operator.

Let \(\mathcal A\subset L^2(\Omega;E)\) be a class of admissible forcing
fields and let \(A\in\overline{\mathcal A}\). Assume that
$
\Sigma_{A'}\preceq \Sigma_A$ for all $A'\in\mathcal A$. The covariance-envelope comparison acts only on the perturbative forcing
component \(A\), while the baseline component \(\xi^A\) contributes a fixed
second-order term independent of the envelope optimization. Since
\(\mathcal S\) is bounded and linear, covariance domination on the
perturbative forcing component is transported through the induced reconstruction
operator:
\[
\mathcal S\Sigma_{A'}\mathcal S^*
\preceq
\mathcal S\Sigma_A\mathcal S^*.
\]

Thus covariance domination at the level of perturbative forcing fields is
preserved under the induced elliptic reconstruction structure.

The extremal principle of the present paper therefore applies to the associated
reconstruction problem. In particular, for every admissible reconstruction
operator \(T\in V_1\),
$
\sup_{A'\in\mathcal A}R_{A'}(T)=R_A(T).
$
\end{ex}

    \section{Proofs}\label{sec:pfs}

\subsection{Proof of Proposition \ref{closed}}
\begin{proof}
Let $\{A_n\}_{n\in\N}\subset C_{\mathcal A}(A)$ and assume that
$A_n\to A'$ in $\mathcal V$.
Since $\mathcal A$ is closed in $\mathcal V$ and $A_n\in\mathcal A$ for all $n$,
we have $A'\in\mathcal A$. Fix $\mathbf g\in \mathbb H_o^{d}$.
By definition of $C_{\mathcal A}(A)$ we have for every $n$,
$
\langle \mathbf g,\Sigma_{A_n}\mathbf g\rangle_{\mathbb H_o^{d}}
\le
\langle \mathbf g,\Sigma_A\mathbf g\rangle_{\mathbb H_o^{d}}.
$

\noindent By the continuity property of the covariance operator,
$A_n\to A'$ in $\mathcal V$ implies
$
\langle \mathbf g,\Sigma_{A_n}\mathbf g\rangle_{\mathbb H_o^{d}}
\longrightarrow
\langle \mathbf g,\Sigma_{A'}\mathbf g\rangle_{\mathbb H_o^{d}}.
$
Passing to the limit in the inequality yields
\[
\langle \mathbf g,\Sigma_{A'}\mathbf g\rangle_{\mathbb H_o^{d}}
\le
\langle \mathbf g,\Sigma_A\mathbf g\rangle_{\mathbb H_o^{d}}
\qquad \forall\,\mathbf g\in\mathbb H_o^{d}.
\]
Hence $A'\in C_{\mathcal A}(A)$, and therefore $C_{\mathcal A}(A)$ is closed in $\mathcal V$.
\end{proof}

\subsection{Proof of Proposition \ref{Gprop}}
\begin{proof}
Define
\[
C
:=
\Bigl\{
A'\in\mathcal A:
\langle \mathbf g,\Sigma_{A'}\mathbf g\rangle_{\mathbb H_o^{d}}
\le
\langle \mathbf g,\Sigma_A\mathbf g\rangle_{\mathbb H_o^{d}}
\quad \forall\,\mathbf g\in\mathcal G^{d}
\Bigr\}.
\]
Since $\mathcal G^{d}\subset \mathbb H_o^{d}$, we trivially have
$C_{\mathcal A}(A)\subseteq C$. To prove the reverse inclusion, let $\mathbf g\in\mathbb H_o^{d}$ be arbitrary.
Since $\mathcal G$ is dense in $\mathbb H_o$, there exists a sequence
$\mathbf g_n\in\mathcal G^{d}$ such that
$
\|\mathbf g_n-\mathbf g\|_{\mathbb H_o^{d}}\longrightarrow 0
$. Fix $A'\in C$. By bilinearity of the quadratic form and the triangle inequality,
\begin{align*}
&
\big|
\langle \mathbf g,(\Sigma_A-\Sigma_{A'})\mathbf g\rangle
-
\langle \mathbf g_n,(\Sigma_A-\Sigma_{A'})\mathbf g_n\rangle
\big|
\\
&\le
\big|
\langle \mathbf g-\mathbf g_n,(\Sigma_A-\Sigma_{A'})\mathbf g\rangle
\big|
+
\big|
\langle \mathbf g_n,(\Sigma_A-\Sigma_{A'})(\mathbf g-\mathbf g_n)\rangle
\big|
\\
&\le
2\,\|\Sigma_A-\Sigma_{A'}\|_{\mathcal L(\mathbb H_o^{d})}\,
\|\mathbf g_n-\mathbf g\|_{\mathbb H_o^{d}}\,
\|\mathbf g\|_{\mathbb H_o^{d}} .
\end{align*}
Since $\Sigma_A-\Sigma_{A'}$ is a bounded operator on $\mathbb H_o^{d}$,
the right-hand side converges to zero as $n\to\infty$. By assumption,
\[
\langle \mathbf g_n,(\Sigma_A-\Sigma_{A'})\mathbf g_n\rangle \ge 0
\quad \text{for all } n,
\]
and hence, passing to the limit,
\[
\langle \mathbf g,(\Sigma_A-\Sigma_{A'})\mathbf g\rangle
=
\lim_{n\to\infty}
\langle \mathbf g_n,(\Sigma_A-\Sigma_{A'})\mathbf g_n\rangle
\ge 0 .
\]
Since $\mathbf g\in\mathbb H_o^{d}$ was arbitrary, this shows
$A'\in C_{\mathcal A}(A)$, and therefore $C\subseteq C_{\mathcal A}(A)$. The proof is complete.
\end{proof}

\subsection{Proof of Proposition \ref{WSS}}
\begin{proof}
Denote
\[
C=\Bigl\{A'\in \mathcal{A}:\ 
\widehat{K}_{A}(\omega)-\widehat{K}_{A'}(\omega)
\ \text{is positive semidefinite for Lebesgue-a.e. }\omega\Bigr\}.
\]

Fix $A'\in\mathcal A$ and for $1\le i,j\le  d $ let $K_{i,j}$ denote the $(i,j)$--entry of the matrix-valued kernel
$K_A-K_{A'}$, viewed as a function on $\R$.
Let $f,g\in \mathbb H_o=L^2([t_1,t_2])$ and extend them by zero outside $[t_1,t_2]$,
still denoted $f,g$, so that $f,g\in L^2(\R)$.
By the Plancherel theorem we then have
\begin{align*}
\int_{[t_1,t_2]^2} g(s)\,K_{i,j}(s-t)\,f(t)\,ds\,dt
&=
\int_{\R^2} g(s)\,K_{i,j}(s-t)\,f(t)\,ds\,dt
\\
&=
\int_{\R} g(s)\,(K_{i,j}*f)(s)\,ds
\\
&=
\frac{1}{2\pi}\int_{\R} \widehat g(\omega)\,
\widehat{K_{i,j}}(\omega)\,\widehat f(\omega)\,d\omega .
\end{align*}

Therefore, for any $g=(g_1,\ldots,g_{d})$ with $g_i\in\mathbb H_o$,
\begin{align*}
&\int_{[t_1,t_2]^2}
\big(g_1(s),\ldots,g_{d}(s)\big)\,
\big(K_A(s-t)-K_{A'}(s-t)\big)\,
\big(g_1(t),\ldots,g_{d}(t)\big)^{*}\,ds\,dt
\\
&\qquad=
\frac{1}{2\pi}\int_{\R}
\big(\widehat g_1(\omega),\ldots,\widehat g_{d}(\omega)\big)\,
\big(\widehat K_A(\omega)-\widehat K_{A'}(\omega)\big)\,
\big(\widehat g_1(\omega),\ldots,\widehat g_{d}(\omega)\big)^{*}\,d\omega ,
\end{align*}
where ${}^*$ denotes conjugate transpose.
Hence, if $A'\in C$, then the right-hand side is nonnegative for all $g$,
and therefore $A'\in C_{\mathcal A}(A)$, i.e.\ $C\subseteq C_{\mathcal A}(A)$.

\vspace{0.5em}
\noindent
Conversely, suppose that $A'\in C^c$.
Let $\lambda_{d}(\omega)$ denote the smallest eigenvalue of
$\widehat K_A(\omega)-\widehat K_{A'}(\omega)$.
By assumption there exists a set $D\subset\R$ of positive Lebesgue measure such that
$\lambda_{d}(\omega)<0$ for $\omega\in D$.
Fix $\omega'\in D$ and let $x\in\C^{d}$ be a corresponding unit eigenvector.
By continuity of the entries of
$\widehat K_A(\omega)-\widehat K_{A'}(\omega)$
there exist $\xi>0$ such that
\[
x^*\big(\widehat K_A(\omega)-\widehat K_{A'}(\omega)\big)x<0
\quad\text{for all }\omega\in(\omega'-\xi,\omega'+\xi).
\]

Let $\psi\in C_c^\infty(\R)$ satisfy
$0\le \psi\le 1$,
$\psi\equiv 1$ on $[\omega'-\xi/2,\omega'+\xi/2]$,
and $\psi\equiv 0$ outside
$[\omega'-\xi/2-\delta,\omega'+\xi/2+\delta]$ for some $\delta>0$.
For $\delta$ sufficiently small we then have
\[
\int_{\R} \psi(\omega)^2\,
x^*\big(\widehat K_A(\omega)-\widehat K_{A'}(\omega)\big)x\,d\omega<0.
\]

By the Plancherel theorem,
\begin{align*}
\int_{\R} \psi(\omega)^2\,
x^*\big(\widehat K_A(\omega)-\widehat K_{A'}(\omega)\big)x\,d\omega
&=
2\pi
\int_{\R^2}
\check\psi(s)\,
x^*\big(K_A(s-t)-K_{A'}(s-t)\big)x\,
\check\psi(t)\,ds\,dt ,
\end{align*}
where $\check\psi$ denotes the inverse Fourier transform of $\psi$.
Define
\[
g_i(s):=x_i\,\check\psi(s)\,\mathbf 1_{[t_1,t_2]}(s),
\qquad 1\le i\le  d ,
\]
so that $g_i\in\mathbb H_o$.
Then
\[
\int_{[t_1,t_2]^2}
\big(g_1(s),\ldots,g_{d}(s)\big)\,
\big(K_A(s-t)-K_{A'}(s-t)\big)\,
\big(g_1(t),\ldots,g_{d}(t)\big)^{*}\,ds\,dt
<0,
\]
which implies $A'\notin C_{\mathcal A}(A)$.
Hence $C_{\mathcal A}(A)\subseteq C$, and the proof is complete.
\end{proof}

	\subsection{Proof of Theorem \ref{WR1}}

    We now proceed with the proof of the main Theorem.

	\begin{proof}[Proof of Theorem \ref{WR1}]
		\textbf{Step 1: Reduction to coefficient form.}
        
We first express the cost functional in terms of coefficient expansions of the target and auxiliary components. Take $A'\in\mathcal{V}$. Since $\n \mathcal{S}_Y \n = \n \mathcal{S}_Y \n \le \n \mathcal{S}\n$, 
		\begin{align}\label{Ybound}
			\int_\Omega \lVert Y^{A'} \rVert_{\mathbb{H}_o}^2d\mu 
			&=
			\int_\Omega \lVert\mathcal{S}_Y\left(A'+\xi^{A'}\right) \rVert_{\mathbb{H}_o}^2d\mu\nonumber
			\\
            &\le
			\n \mathcal{S}_Y \n^2 \int_\Omega \n A'+\xi^{A'}\n_{\mathbb{H}_o^{d}}^2d\mu\nonumber
			\\
			&\le
			\n S \n^2 \int_\Omega \left(2\n A'\n_{\mathbb{H}_o^{d}}^2+\n\xi^{A'}\n_{\mathbb{H}_o^{d}}^2\right)d\mu <\infty,
		\end{align}
		which also implies $Y^{A'}\in \mathbb{H}_o$ $\mu$-a.e.. Take some arbitrary complete ON-basis for $\mathbb{H}_o$, $\{\phi_{n}\}_{n\in\N}$ and some complete ON-basis for for $H$, $\{\psi_{n}\}_{n\in\N}$ and define 
        \[
Z^{A'}_k=\langle Y^{A'},\phi_k\rangle_{\mathbb{H}_o}
\qquad\text{and}\quad
\chi^{A'}_k=\langle \psi_k,X^{A'}\rangle_{\H,\H^*}.
\]
If we let $S_n^{Y^{A'}}=\sum_{k=1}^nZ^{A'}_k\phi_{k}$, then $S_n^{Y^{A'}}\xrightarrow{\mathbb{H}_o}Y^{A'}$ a.s., since $\{\phi_k\}_{k\in\N}$ is an ON-basis. Next, by monotone convergence and the Parseval formula
		\begin{align*}
			 \int_\Omega \lVert Y^{A'}\rVert_{\mathbb{H}_o}^2d\mu 
			&=
			\int_\Omega  \sum_{k=1}^\infty(Z^{A'}_k)^2 d\mu 
			=\sum_{k=1}^\infty\int_\Omega  (Z^{A'}_k)^2 d\mu ,
		\end{align*}
		which implies $\sum_{k=1}^\infty\int_\Omega  (Z^{A'}_k)^2 d\mu <\infty$.  Also, $S_n^{Y^{A'}}\xrightarrow{L^2(\Omega;\mathbb{H}_o)}Y^{A'}$ ,
\begin{align}\label{XL2}
			\lim_{n\to\infty}\int_\Omega  \lVert S_n^{Y^{A'}}-Y^{A'}\rVert_{\mathbb{H}_o}^2d\mu \nonumber
			&=
			\lim_{n\to\infty}\int_\Omega  \lim_{N\to\infty}\left\lVert\sum_{k=1}^nZ^{A'}_k\phi_{k}-\sum_{k=1}^NZ^{A'}_k\phi_{k}\right\rVert_{\mathbb{H}_o}^2 d\mu \nonumber
			\\
			&=
			\lim_{n\to\infty}\int_\Omega  \lim_{N\to\infty}\left\lVert \sum_{k=n+1}^NZ^{A'}_k\phi_{k}\right\rVert_{\mathbb{H}_o}^2  d\mu \nonumber
			\\
			&=
			\lim_{n\to\infty}\int_\Omega  \lim_{N\to\infty}\sum_{k=n+1}^N\left(Z^{A'}_k\right)^2 d\mu \nonumber
			\\
			&=\lim_{n\to\infty}\sum_{k=n+1}^\infty\int_\Omega  \left(Z^{A'}_k\right)^2 d\mu =0,
		\end{align}
		by monotone convergence. Let $\mathcal R:\H\to\H^*, \qquad (\mathcal R h)(\psi):=\langle h,\psi\rangle_{\H}$. By Riesz representation theorem there exists $h\in\H$ (depending on $\omega\in\Omega)$ such that $X^{A'}(\psi)=\langle h,\psi\rangle_\H$ for all $\psi\in\H$, and $\lVert h\rVert_\H=\lVert X^{A'}\rVert_{\H^*}$. 
        By the Parseval theorem,
        \begin{align*}
        \sum_{l=1}^\infty(\chi_l^{A'})^2
        &=\sum_{l=1}^\infty\left\langle \psi_l , X^{A'}\right\rangle_{\H,\H^*}^2
        \\
        &=\sum_{l=1}^\infty\left\langle \psi_l , h\right\rangle_{\H}^2
        =\lVert X^{A'}\rVert_{\H^{*}}^2
        \end{align*}
        and since 
        \begin{align*}
        \int_\Omega  \lVert X^{A'}\rVert_{\H^*}^2d\mu 
        &\le 
        \lVert \mathcal{S}_X\rVert^2 \int_\Omega  \lVert A'+\xi^{A'}\rVert_{\mathbb{H}_o}^2d\mu 
        \\
        &\le \lVert \mathcal{S}\rVert^2 2\left(\lVert \xi^{A'}\rVert_{\mathcal{V}}^2+\lVert A'\rVert_{\mathcal{V}}^2\right)
        \end{align*}
        it follows that 
        \begin{align}\label{chisummable}
        \sum_{k=1}^\infty\int_\Omega (\chi_k^{A'})^2d\mu <\infty.
        \end{align}

\noindent We endow $\H^*$ with its canonical Hilbert space structure via the Riesz
isometric isomorphism
\[
\mathcal R:\H\to\H^*, \qquad (\mathcal R u)(\psi):=\langle u,\psi\rangle_{\H}.
\]
Via this identification, $\mathcal{HS}(\H^*,\mathbb H_o)$ is canonically
isometrically isomorphic to $\mathcal{HS}(\H,\mathbb H_o)$, and hence to
$\mathbb H_o\widehat\otimes \H$. Let $\{\phi_k\}_{k\ge1}$ and $\{\psi_\ell\}_{\ell\ge1}$ be orthonormal bases of
$\mathbb{H}_o$ and $\H$, respectively. As $\mathcal{HS}(\H^*,\mathbb{H}_o)\cong \mathbb{H}_o \widehat{\otimes}\H$, 
$\{\phi_k\otimes\psi_\ell\}_{k,\ell\ge1}$ is an orthonormal basis of
$\mathbb{H}_o \widehat{\otimes}\H$, and hence
\[
\mathbb{H}_o \widehat{\otimes}\H
=
\left\{
\sum_{k=1}^\infty\sum_{\ell=1}^\infty \lambda_{k,\ell}\,\phi_k\otimes\psi_\ell
:\ \sum_{k,\ell}\lambda_{k,\ell}^2<\infty
\right\},
\]
with convergence in $\lVert\cdot\rVert_{\mathbb{H}_o \widehat{\otimes}\H}$.

\noindent Under the canonical isometric identification
\[
\mathcal I:\mathbb H_o\widehat\otimes \H \longrightarrow \mathcal{HS}(\H^*,\mathbb H_o),
\]
let $\beta_T\in \mathbb H_o\widehat\otimes \H$ denote the unique element such that
$\mathcal I(\beta_T)=T$.
Writing
\[
\beta_T=\sum_{k=1}^\infty\sum_{\ell=1}^\infty \lambda_{k,\ell}^{\beta_T}\,\phi_k\otimes\psi_\ell,
\qquad
\lambda_{k,\ell}^{\beta_T}:=\langle \beta_T,\phi_k\otimes\psi_\ell\rangle_{\mathbb H_o\widehat\otimes \H},
\]
we have for every $x^*\in\H^*$,
\[
T(x^*)
=
\sum_{k=1}^\infty\sum_{\ell=1}^\infty
\lambda_{k,\ell}^{\beta_T}\,\langle \psi_\ell,x^*\rangle_{\H,\H^*}\,\phi_k,
\]
with convergence in $\mathbb H_o$.
In particular, since $\chi_\ell^{A'}=\langle \psi_\ell,X^{A'}\rangle_{\H,\H^*}$,
\[
T(X^{A'})
=
\sum_{k=1}^\infty\sum_{\ell=1}^\infty
\lambda_{k,\ell}^{\beta_T}\,\chi_\ell^{A'}\,\phi_k.
\]

\noindent For $\beta_T$ define the partial sums
\[
S_n^{\beta_T}
:=
\sum_{k=1}^n\sum_{\ell=1}^n \lambda_{k,\ell}^{\beta_T}\,\phi_k\otimes\psi_\ell.
\]
Then $\left\lVert S_n^{\beta_T}-\beta_T\right\rVert_{\mathbb{H}_o \widehat{\otimes}\H}\to 0$ as $n\to\infty$. Next, we note that, 
\begin{align}\label{assbound}
\int_\Omega \left\lVert T( X^{A'})\right\rVert_{\mathbb{H}_o}^2d\mu
&\le
\int_\Omega \lVert T \rVert_{V_1}^2 \lVert X^{A'} \rVert_{\H^*}^2 d\mu \nonumber
\\
&\le \lVert T \rVert_{V_1}^2  \int_\Omega \left(2\lVert  \mathcal{S}_{X} A' \rVert_{\mathbb{H}_o^{d}}^2+2\lVert \mathcal{S}_{X}\xi^{A'} \rVert_{\mathbb{H}_o^{d}}^2\right)d\mu \nonumber
\\
&\le 2\lVert T \rVert_{V_1}^2 \lVert \mathcal{S}\rVert^2 \int_\Omega \left(\lVert A' \rVert_{\mathbb{H}_o^{d}}^2+\lVert \xi^{A'} \rVert_{\mathbb{H}_o^{d}}^2\right)d\mu <\infty,
\end{align}
where we utilized that 
$$\left\lVert T( X^{A'})\right\rVert_{\mathbb{H}_o}\le  \lVert T \rVert \lVert X^{A'} \rVert_{\H^*} \le  \lVert T \rVert_{V_1} \lVert X^{A'} \rVert_{\H^*}.$$
Therefore if we let $Q=T\left(X^{A'} \right) $ and 
$$S_n^{\int}=\sum_{k=1}^n \left\langle  T( X^{A'}),\phi_k\right\rangle_{\mathbb{H}_o} \phi_k$$ 
then $S_n^{\int}\xrightarrow{\mathbb{H}_o}Q$ $\mu$-a.e.. Since
\begin{align*}
\left\langle T(X^{A'}),\phi_k\right\rangle_{\mathbb{H}_o}
&=
\left\langle
\sum_{k'=1}^\infty\sum_{\ell=1}^\infty
\lambda_{k',\ell}^{\beta_T}\,
\pairHH{\psi_\ell}{X^{A'}}\,
\phi_{k'},
\;\phi_k
\right\rangle_{\mathbb{H}_o}
\\
&=
\sum_{k'=1}^\infty\sum_{\ell=1}^\infty
\lambda_{k',\ell}^{\beta_T}\,
\pairHH{\psi_\ell}{X^{A'}}\,
\langle \phi_{k'},\phi_k\rangle_{\mathbb{H}_o}
\\
&=
\sum_{\ell=1}^\infty
\lambda_{k,\ell}^{\beta_T}\,
\pairHH{\psi_\ell}{X^{A'}}
=
\sum_{\ell=1}^\infty
\lambda_{k,\ell}^{\beta_T}\,\chi_\ell^{A'} .
\end{align*}

we get,
$
S_n^{\int}
=
\sum_{k=1}^n\sum_{l=1}^\infty \lambda_{k,l}^{\beta_T} \chi_l^{A'}  \phi_k.
$
Next, utilizing orthonormality
\begin{align}\label{series1}
\int_\Omega \left\lVert\sum_{k=1}^n\sum_{l=1}^n \lambda^{\beta_T}_{k,l}\chi^{A'}_l\phi_k- T\left(X^{A'} \right) \right\rVert_{\mathbb{H}_o}^2 d\mu 
&\le
\int_\Omega\left\lVert S_n^{\int}-\sum_{k=1}^n\sum_{l=n+1}^\infty \lambda_{k,l}^{\beta_T} \chi_l^{A'}  \phi_k- Q \right\rVert_{\mathbb{H}_o}^2 d\mu \nonumber
\\
&\le
\int_\Omega\left\lVert S_n^{\int}- Q\right\rVert_{\mathbb{H}_o}^2 d\mu 
+
\int_\Omega\left\lVert \sum_{k=1}^n\sum_{l=n+1}^\infty \lambda_{k,l}^{\beta_T} \chi_l^{A'}  \phi_k \right\rVert_{\mathbb{H}_o}^2 d\mu 
\nonumber
\\
&=
\int_\Omega \left\lVert S_n^{\int}- Q \right\rVert_{\mathbb{H}_o}^2 d\mu 
+
\sum_{k=1}^n  \int_\Omega \left( \sum_{l=n+1}^\infty  \lambda_{k,l}^{\beta_T}\chi_l^{A'}  \right)^2 d\mu 
\nonumber
\\
&\le
\int_\Omega \left\lVert S_n^{\int}- Q \right\rVert_{\mathbb{H}_o}^2 d\mu 
+
\sum_{k=1}^\infty \sum_{l=1}^\infty  \left(\lambda_{k,l}^{\beta_T}\right)^2\int_\Omega  \sum_{j=n+1}^\infty  \left(\chi_j^{A'}\right)^2   d\mu 
\nonumber
\\
&=
\int_\Omega \left\lVert S_n^{\int}- Q \right\rVert_{\mathbb{H}_o}^2 d\mu 
+
 \lVert T \rVert_{V_1}^2 \int_\Omega  \sum_{j=n+1}^\infty  \left(\chi_j^{A'}\right)^2   d\mu
\end{align}
where, by the Cauchy-Schwarz inequality and the definition of $\beta_T$
\begin{align}\label{finitdubbelsumma}
    \sum_{k=1}^\infty\int_\Omega \left( \sum_{l=n+1}^\infty  \lambda_{k,l}^{\beta_T}\chi_l^{A'}  \right)^2 d\mu 
    &\le 
    \sum_{k=1}^\infty \sum_{l=1}^\infty  \left(\lambda_{k,l}^{\beta_T}\right)^2\int_\Omega  \sum_{j=n+1}^\infty  \left(\chi_j^{A'}\right)^2   d\mu \nonumber
    \\
    &= \lVert T \rVert_{V_1}^2 \int_\Omega  \sum_{j=n+1}^\infty  \left(\chi_j^{A'}\right)^2   d\mu.
\end{align}
The second term on the right-most side of \eqref{series1} converges to zero due to \eqref{chisummable}. We now wish to bound the first term on the right-most side of \eqref{series1} using \eqref{assbound},
\begin{align}
M_n:=\left\lVert S_n^{\int}- Q \right\rVert_{\mathbb{H}_o}^2\nonumber
&\le
2\left\lVert  Q \right\rVert_{\mathbb{H}_o}^2
+
2\left\lVert  S_n^{\int} \right\rVert_{\mathbb{H}_o}^2\nonumber
\\
&\le
2\lVert  T \rVert_{V_1}^2 \lVert \mathcal{S} \rVert^2 \left(2\lVert A' \rVert_{\mathbb{H}_o^{d}}^2+2\lVert \xi^{A'} \rVert_{\mathbb{H}_o^{d}}^2\right)
+
2\left\lVert \sum_{k=1}^n\sum_{l=1}^\infty \lambda_{k,l}^{\beta_T} \chi_l^{A'}\phi_k\right\rVert_{\mathbb{H}_o}^2\nonumber
\\
&=4\lVert  T \rVert_{V_1}^2  \lVert \mathcal{S} \rVert^2 \left(\lVert A' \rVert_{\mathbb{H}_o^{d}}^2+\lVert \xi^{A'} \rVert_{\mathbb{H}_o^{d}}^2\right)+
2\sum_{k=1}^n\left(\sum_{l=1}^\infty \lambda_{k,l}^{\beta_T} \chi_l^{A'}  \right)^2
\end{align}
and therefore
$$M_n\le 4\lVert  T \rVert_{V_1}^2  \lVert \mathcal{S} \rVert^2 \left(\lVert A' \rVert_{\mathbb{H}_o^{d}}^2+\lVert \xi^{A'} \rVert_{\mathbb{H}_o^{d}}^2\right)
+
2\sum_{k=1}^\infty\left(\sum_{l=1}^\infty \lambda_{k,l}^{\beta_T} \chi_l^{A'}  \right)^2:=M.$$
Combining \eqref{assbound} and \eqref{finitdubbelsumma} (with $n=0$) implies $\int_\Omega Md\mu <\infty$. Since $\{M_n\}_{n\in\N}$ converges to zero $\mu$-a.e. and $0\le M_n\le M$ it follows from the dominated convergence theorem that 
$$\lim_{n\to\infty}\int_\Omega \left\lVert S_n^{\int}- Q\right\rVert_{\mathbb{H}_o}^2 d\mu =0
$$ 
and therefore due to \eqref{series1} we get 
\begin{align}\label{series}
\lim_{n\to\infty}\int_\Omega  \left\lVert \sum_{k=1}^n\sum_{l=1}^n \lambda^{\beta_T}_{k,l}\chi^{A'}_l\phi_k-T\left(X^{A'} \right)  \right\rVert_{\mathbb{H}_o}^2 d\mu 
=0.
\end{align}
Using the Cauchy-Schwarz and the reverse triangle inequality we find,
        \begin{align}\label{RAEQ}
&\Biggl|
R_{A'}(T)
-
\lim_{n\to\infty}
\int_\Omega
   \left\lVert
         S_n^{Y^{A'}}
         -
         \sum_{k,l=1}^n 
            \lambda^{\beta_T}_{k,l}\,
            \chi^{A'}_l\,
            \phi_k\right\rVert_{\mathbb{H}_o}^2
d\mu
\Biggr|\nonumber
\\[0.6em]
&=
\Biggl|
\int_\Omega
\left\lVert
         Y^{A'}
         -
            T\left(X^{A'} \right) \right\rVert_{\mathbb{H}_o}^2
d\mu-
\lim_{n\to\infty}
\int_\Omega
   \left\lVert
         S_n^{Y^{A'}}
         -
         \sum_{k,l=1}^n 
            \lambda^{\beta_T}_{k,l}\,
            \chi^{A'}_l\,
            \phi_k
\right\rVert_{\mathbb{H}_o}^2
d\mu
\Biggr|
\nonumber
\\[0.6em]
&\le
\lim_{n\to\infty} \int_\Omega \Biggl|
\left\lVert
         Y^{A'}
         -
            T\left(X^{A'} \right) \right\rVert_{\mathbb{H}_o}
+
   \left\lVert
         S_n^{Y^{A'}}
         -
         \sum_{k,l=1}^n 
            \lambda^{\beta_T}_{k,l}\,
            \chi^{A'}_l\,
            \phi_k
\right\rVert_{\mathbb{H}_o}
\Biggr|\nonumber
\\[-0.3em]
&\qquad\qquad\qquad\qquad\qquad\qquad\cdot
\Biggl|
\left\lVert
         Y^{A'}
         -
            T\left(X^{A'} \right) \right\rVert_{\mathbb{H}_o}
-
   \left\lVert
         S_n^{Y^{A'}}
         -
         \sum_{k,l=1}^n 
            \lambda^{\beta_T}_{k,l}\,
            \chi^{A'}_l\,
            \phi_k
\right\rVert_{\mathbb{H}_o}
\Biggr|d\mu
\nonumber
\\[0.7em]
&\le
\lim_{n\to\infty}
\int_\Omega
        \left(\underbrace{\left\lVert
        Y^{A'}
        -
            T\left(X^{A'} \right)  \right\rVert_{\mathbb{H}_o}
        + \left\lVert\sum_{k,l=1}^n
              \lambda^{\beta_T}_{k,l}\chi^{A'}_l\phi_k
        + S_n^{Y^{A'}}
         \right\rVert_{\mathbb{H}_o} }_{=:D_n}
      \right)\nonumber
\\[-0.3em]
&\qquad\qquad\qquad\qquad\qquad\qquad\cdot
        \underbrace{
        \left\lVert
        Y^{A'}
        -
            T\left(X^{A'} \right) 
        +\sum_{k,l=1}^n
              \lambda^{\beta_T}_{k,l}\chi^{A'}_l\phi_k
        - S_n^{Y^{A'}}
        \right\rVert_{\mathbb{H}_o}
        }_{=:D_n^\sharp}
d\mu
\nonumber
\\[0.6em]
&\le
\lim_{n\to\infty}
\left(\int_\Omega D_n^2d\mu\right)^{1/2}
\,
\left(\int_\Omega (D_n^\sharp)^2d\mu\right)^{1/2}.
\end{align}
For the first factor on the right-most side above we apply the Pythagoran theorem and the elementary inequality $(a+b+c+d)^2\le 4(a^2+b^2+c^2+d^2)$ 
\begin{align*}
\int_\Omega  D_n^2d\mu 
&=
\int_\Omega \left(\left\lVert
        Y^{A'}
        -
            T\left(X^{A'} \right)  \right\rVert_{\mathbb{H}_o}
        + \left\lVert\sum_{k,l=1}^n
              \lambda^{\beta_T}_{k,l}\chi^{A'}_l\phi_k
        + S_n^{Y^{A'}}\right\rVert_{\mathbb{H}_o}\right)^2d\mu 
\\
&\le 
4\int_\Omega \lVert Y^{A'} \rVert_{\mathbb{H}_o}^2d\mu
+4\int_\Omega \left\lVert T\left(X^{A'} \right) \right\rVert_{\mathbb{H}_o}^2d\mu 
+4\sum_{k=1}^n\int_\Omega \left(\sum_{l=1}^n \lambda^{\beta_T}_{k,l}\chi^{A'}_l\right)^2d\mu 
+4\int_\Omega \lVert Y^{A'} \rVert_{\mathbb{H}_o}^2d\mu
\\
&\le
8\int_\Omega \lVert Y^{A'} \rVert_{\mathbb{H}_o}^2d\mu
+4\lVert  T \rVert_{V_1}^2\lVert \mathcal{S}\rVert^2\left( \lVert A'\rVert_{\mathcal V}+\lVert \xi^{A'}\rVert_{\mathcal V}\right)^2
+4\sum_{k=1}^\infty \sum_{l=1}^\infty  \left(\lambda_{k,l}^{\beta_T}\right)^2\int_\Omega  \sum_{j=1}^\infty  \left(\chi_j^{A'}\right)^2   d\mu ,
\end{align*}
where the term final term on right-most side is finite (and so the left-most side is uniformly bounded in $n$) due to \eqref{chisummable}, \eqref{Ybound} and the fact that 
$$\sum_{k=1}^\infty \sum_{l=1}^\infty \left(\lambda_{k,l}^{\beta_T}\right)^2= \lVert \beta_T\rVert_{\mathbb{H}_o \widehat{\otimes}\H}^2=\lVert T\rVert_{V_1}^2,$$ 
so the first factor in \eqref{RAEQ}.
For the second factor on the right-most side of \eqref{RAEQ}
\begin{align}
    \int_\Omega (D_n^\sharp)^2d\mu
    &=
    \int_\Omega \left\lVert \left(Y^{A'}-S_n^{Y^{A'}}\right)+\left(\sum_{k=1}^n\sum_{l=1}^n \lambda^{\beta_T}_{k,l}\chi^{A'}_l\phi_k-T\left(X^{A'} \right)\right)\right\rVert_{\mathbb{H}_o}^2d\mu \nonumber
    \\
    &\le
    2\int_\Omega\left\lVert Y^{A'}-S_n^{Y^{A'}} \right\rVert_{\mathbb{H}_o}^2 d\mu \nonumber
+
    2\int_\Omega \left\lVert T\left(X^{A'}\right)- \sum_{k=1}^n\sum_{l=1}^n \lambda^{\beta_T}_{k,l}\chi^{A'}_l\phi_k\right\rVert_{\mathbb{H}_o}^2 d\mu 
\end{align}
where the first term converges to zero due to the fact that $S_n^{Y^{A'}}\xrightarrow{L^2(\Omega;\mathbb{H}_o)}Y^{A'}$ (i.e. \eqref{XL2}) and the second term converges to zero by \eqref{series}. We conclude that the right-most side of \eqref{RAEQ} converges to zero. Therefore
		\begin{align}\label{RAlim}
			R_{A'}(T)
			&=
			\lim_{n\to\infty}\int_\Omega \left\lVert \sum_{k=1}^n Z_k^{A'}\phi_k-\sum_{k=1}^n\sum_{l=1}^n \lambda^{\beta_T}_{k,l}\chi^{A'}_l\phi_k\right\rVert_{\mathbb{H}_o}^2d\mu \nonumber
			\\
			&=
			\lim_{n\to\infty}\sum_{k=1}^n\int_\Omega  \left(Z_k^{A'}- \sum_{l=1}^n \lambda_{k,l}^{\beta_T}\chi_l^{A'}\right)^2d\mu .
		\end{align}


\noindent\textbf{Step 2: Reformulate the the integrals appearing in \eqref{ShiftSet} for relevant subspaces}\\
For $W\in \mathbb{H}_o$ define the coefficient functionals
\[
F_k(W):=\langle W,\phi_k\rangle_{ \mathbb{H}_o},\qquad k\ge1.
\]
For $n\in\mathbb N$ let $P_n: \mathbb{H}_o\to \mathbb{H}_o$ denote the orthogonal projection onto
$\operatorname{span}\{\phi_1,\ldots,\phi_n\}$.

Let $A'=(A'(1),\ldots,A'(d))$ be an $\mathbb{H}_o^{d}$--valued random element with finite second moment.
We write $\mathbf P_n:=\operatorname{diag}(P_n,\ldots,P_n)$ for the induced projection on $\mathbb{H}_o^{d}$.
Define the finite--dimensional coefficient vector
\[
F_{1:n}(A')
:=
\big(F_1(A'(1)),\ldots,F_n(A'(1)),\; \ldots,\; F_1(A'(d)),\ldots,F_n(A'(d))\big)^\top
\in\mathbb R^{dn}.
\]
For $1\le i,j\le  d $ define the second-moment operators
\[
\Sigma_{ij}:=\int_\Omega A'(i)\otimes A'(j)d\mu\;:\;\mathbb{H}_o\to\mathbb{H}_o,
\]
where $(x\otimes y)u := \langle u,y\rangle_{\mathbb{H}_o}\,x$.
Collecting the blocks yields the operator
\[
\Sigma_{A'}
:=
\begin{bmatrix}
\Sigma_{11} & \cdots & \Sigma_{1, d }\\
\vdots      & \ddots & \vdots\\
\Sigma_{ d ,1} & \cdots & \Sigma_{ d , d }
\end{bmatrix}
\;:\; \mathbb{H}_o^{d}\to\mathbb{H}_o^{d}.
\]
Define the finite--rank compression
\[
\Sigma_{A'}^{(n)}:=\mathbf P_n\,\Sigma_{A'}\,\mathbf P_n.
\]
Let $u,v\in\mathbb{H}_o$ be arbitrary. By definition of $\Sigma_{ij}^{(n)}$ and since $P_n$ is self--adjoint,
\[
\langle \Sigma_{ij}^{(n)}u,v\rangle_{\mathbb{H}_o}
=
\langle \Sigma_{ij}P_n u,P_n v\rangle_{\mathbb{H}_o}.
\]
Note that $\lVert A'(i)\otimes A'(j)\rVert_{\mathcal{HS}(\mathbb{H}_o)}\le \lVert A'(i)\rVert_{\mathbb{H}_o}\lVert A'(j)\rVert_{\mathbb{H}_o}$ and therefore $A'(i)\otimes A'(j)$ defines a Hilbert-Schmidt operator on $\mathbb{H}_o$ and therefore we obtain
\begin{align*}
\langle \Sigma_{ij}P_n u,P_n v\rangle_{\mathbb{H}_o}
&=
\left\langle \int_\Omega (A'(i)\otimes A'(j))d\mu P_n u,P_n v\right\rangle_{\mathbb{H}_o}
\\
&=
\int_\Omega
\langle (A'(i)\otimes A'(j))P_n u,P_n v\rangle_{\mathbb{H}_o}
d\mu 
\end{align*}
Since $(x\otimes y)z=\langle z,y\rangle_{\mathbb{H}_o}x$, this becomes
\[
\int_\Omega
\langle P_n u,A'(j)\rangle_{\mathbb{H}_o}\;
\langle A'(i),P_n v\rangle_{\mathbb{H}_o}
\,d\mu .
\]
Next, expand the orthogonal projections with respect to the orthonormal basis
$\{\phi_k\}_{k\ge1}$:
\[
P_n u=\sum_{\ell=1}^n \langle u,\phi_\ell\rangle_{\mathbb{H}_o}\,\phi_\ell,
\qquad
P_n v=\sum_{k=1}^n \langle v,\phi_k\rangle_{\mathbb{H}_o}\,\phi_k .
\]
Hence,
\[
\langle P_n u,A'(j)\rangle_{\mathbb{H}_o}
=
\sum_{\ell=1}^n \langle u,\phi_\ell\rangle_{\mathbb{H}_o}\,
\langle \phi_\ell,A'(j)\rangle_{\mathbb{H}_o}
=
\sum_{\ell=1}^n \langle u,\phi_\ell\rangle_{\mathbb{H}_o}\,F_\ell(A'(j)),
\]
and analogously
\[
\langle A'(i),P_n v\rangle_{\mathbb{H}_o}
=
\sum_{k=1}^n \langle v,\phi_k\rangle_{\mathbb{H}_o}\,F_k(A'(i)).
\]
Multiplying these expressions yields
\[
\langle P_n u,A'(j)\rangle_{\mathbb{H}_o}\;
\langle A'(i),P_n v\rangle_{\mathbb{H}_o}
=
\sum_{k=1}^n\sum_{\ell=1}^n
\langle u,\phi_\ell\rangle_{\mathbb{H}_o}\,
\langle v,\phi_k\rangle_{\mathbb{H}_o}\,
F_k(A'(i))\,F_\ell(A'(j)).
\]
Interchanging summation and integration, we obtaining
\[
\langle \Sigma_{ij}^{(n)}u,v\rangle_{\mathbb{H}_o}
=
\sum_{k=1}^n\sum_{\ell=1}^n
\left(
\int_\Omega F_k(A'(i))\,F_\ell(A'(j))\,d\mu
\right)
\langle u,\phi_\ell\rangle_{\mathbb{H}_o}\,
\langle v,\phi_k\rangle_{\mathbb{H}_o}.
\]

Finally, observe that for each $k,\ell$,
\[
\langle (\phi_k\otimes\phi_\ell)u , v\rangle_{\mathbb{H}_o}
=
\langle u,\phi_\ell\rangle_{\mathbb{H}_o}\,
\langle v,\phi_k\rangle_{\mathbb{H}_o}.
\]
Therefore,
\[
\langle \Sigma_{ij}^{(n)}u,v\rangle_{\mathbb{H}_o}
=
\left\langle
\sum_{k=1}^n\sum_{\ell=1}^n
\left(
\int_\Omega F_k(A'(i))\,F_\ell(A'(j))\,d\mu
\right)
(\phi_k\otimes\phi_\ell)u,
\;v
\right\rangle_{\mathbb{H}_o}.
\]
Since this holds for all $u,v\in\mathbb{H}_o$, we conclude that
\[
\Sigma_{ij}^{(n)}
=
\sum_{k=1}^n\sum_{\ell=1}^n
\left(
\int_\Omega F_k(A'(i))\,F_\ell(A'(j))\,d\mu
\right)
(\phi_k\otimes\phi_\ell).
\]
where $(\phi_k\otimes\phi_\ell)u:=\langle u,\phi_\ell\rangle_{\mathbb{H}_o}\,\phi_k$.
Let $\mathbf g=(g_1,\ldots,g_{d})\in\mathbb{H}_o^{d}$ with $g_i\in\operatorname{Ran}(P_n)$ for each $i$.
Then $\mathbf P_n\mathbf g=\mathbf g$, and therefore
\begin{align*}
\langle \mathbf g,\Sigma_{A'}\mathbf g\rangle_{\mathbb{H}_o^{d}}
&=
\langle \mathbf P_n \mathbf g,\Sigma_{A'} \mathbf P_n\mathbf g\rangle_{\mathbb{H}_o^{d}}
\\
&=
\langle \mathbf  \mathbf g,\mathbf P_n^* \Sigma_{A'} \mathbf P_n\mathbf g\rangle_{\mathbb{H}_o^{d}}
\\
&=
\langle \mathbf  \mathbf g,\mathbf P_n \Sigma_{A'} \mathbf P_n\mathbf g\rangle_{\mathbb{H}_o^{d}}
=
\langle \mathbf g,\Sigma_{A'}^{(n)}\mathbf g\rangle_{\mathbb{H}_o^{d}}.
\end{align*}

\textbf{Step 3: Finite–dimensional approximation of the target and auxiliary components.}
\\[0.4em]

We equip the output space $\mathbb H_o\times \H^*$ with the product pairing against
$\mathbb H_o\times \H$,
\[
\big\langle (y,x),\ (f,g)\big\rangle_{\mathrm{out}}
:= \langle y,f\rangle_{\mathbb H_o}+\langle g,x\rangle_{\H,\H^*}.
\]

Fix orthonormal bases $\{\phi_k\}_{k\ge1}$ of $\mathbb H_o$ and
$\{\psi_\ell\}_{\ell\ge1}$ of $\H$.
For $(y,x)\in\mathbb H_o\times\H^*$ define the truncated output–coordinate map
\[
\Pncoef(y,x)
:=
\big(
\langle y,\phi_1\rangle_{\mathbb H_o},\ldots,\langle y,\phi_n\rangle_{\mathbb H_o},
\ \langle \psi_1,x\rangle_{\H,\H^*},\ldots,\langle \psi_n,x\rangle_{\H,\H^*}
\big)^\top
\in\mathbb R^{2n}.
\]

On the input side, let
\[
V_n
:=
\mathrm{span}\big\{(\phi_k,0,\ldots,0),\ldots,(0,\ldots,0,\phi_k):1\le k\le n\big\}
\subset \mathbb H_o^{d},
\]
and let $H_n:V_n\to\mathbb R^{dn}$ denote the coordinate map in the canonical
basis of $V_n$.
Let $P_n^{\mathrm{in}}$ be the orthogonal projection onto $V_n$
(applied componentwise in $\mathbb H_o^{d}$).

We define the finite–dimensional matrix
\[
B^n\in\mathbb R^{2n\times dn}
\]
by the relation
\[
\boxed{\qquad
\Pncoef\!\big(\mathcal{S}\,x\big)
\;=\;
B^n\,H_n(x),
\qquad x\in V_n.
\qquad}
\]

Equivalently, if $\{e_j\}_{j=1}^{dn}$ denotes the canonical input basis of $V_n$,
then
\begin{align*}
(B^n)_{k,j}
&=
\big\langle \mathcal{S}_Y(e_j),\phi_k\big\rangle_{\mathbb H_o},
&& 1\le k\le n,\\
(B^n)_{n+\ell,\,j}
&=
\big\langle \psi_\ell,\mathcal{S}_X(e_j)\big\rangle_{\H,\H^*},
&& 1\le \ell\le n.
\end{align*}

For $A',\xi^{A'}\in\mathcal V$ define the truncated output coefficient vectors
\[
\mathbf Z^{\,n}
:=
(Z^{A'}_1,\ldots,Z^{A'}_n)^\top\in\mathbb R^n,
\qquad
\boldsymbol\chi^{\,n}
:=
(\chi^{A'}_1,\ldots,\chi^{A'}_n)^\top\in\mathbb R^n,
\]
where
\[
Z^{A'}_k:=\langle Y^{A'},\phi_k\rangle_{\mathbb H_o},
\qquad
\chi^{A'}_\ell:=\langle \psi_\ell,X^{A'}\rangle_{\H,\H^*}.
\]

By construction of $B^n$ we have
\[
B^n\Big(F_{1:n}(A')+F_{1:n}(\xi^{A'})\Big)
=
\Pncoef\!\big(\mathcal{S}\,P_n^{\mathrm{in}}(A'+\xi^{A'})\big),
\]
and hence
\small
\begin{equation}\label{SEMA}
\begin{bmatrix}
Z^{A'}_1\\ \vdots\\ Z^{A'}_n\\
\chi^{A'}_1\\ \vdots\\ \chi^{A'}_n
\end{bmatrix}
=
B^n
\Big(
F_{1:n}(A')+F_{1:n}(\xi^{A'})
\Big)
+\delta_n(A'),
\end{equation}
\normalsize
where the truncation error is
\begin{equation}\label{delta}
\delta_n(A')
:=
\Pncoef\!\big(
\mathcal{S}(A'+\xi^{A'})
-
\mathcal{S}\,P_n^{\mathrm{in}}(A'+\xi^{A'})
\big)
\in\mathbb R^{2n}.
\end{equation}

Finally, for any $a\in \mathbb{H}_o^{d}$, trivially we have
\begin{align}\label{PnS}
\lim_{n\to\infty}\lVert a-P_n^{\mathrm{in}}a \rVert_{\mathbb{H}_o^{d}}=0,
\end{align}
since $\{\phi_k\}_{k\ge1}$ is an orthonormal basis of $\mathbb H_o$ and
$V_n$ exhausts $\mathbb H_o^{d}$.

		\noindent\textbf{Step 4:  Approximate the cost using the finite dimensional approximation from the previous step}
		\\
		For $1\le k\le n$, let $\textbf{v}_{n,k}=B^n_{k,.}-\sum_{l=1}^{n}\lambda_{k,l}^{\beta}B^n_{n+l,.}$. From \eqref{RAlim} we have that	for any $A''\in \mathcal{V}$,
		\begin{align}\label{R_A}
			R_{A''}(T) 
			&=
			\lim_{n\to\infty}\sum_{k=1}^n\int_\Omega  \left(Z_k^{A''}- \sum_{l=1}^n \lambda_{k,l}^{\beta}\chi_l^{A''}\right)^2d\mu \nonumber
			\\
			&=
			\lim_{n\to\infty}\sum_{k=1}^n\int_\Omega  \bigg(B^n_{k,.}(F_{1:n}(A'')+F_{1:n}(\xi^{A''}))+(\delta_n(A''))(k)
			\nonumber
			\\
			&- \sum_{l=1}^n\lambda_{k,l}^{\beta}\bigg(B^n_{n+l,.}(F_{1:n}(A'')+F_{1:n}(\xi^{A''})) +(\delta_n(A''))(n+l)\bigg) \bigg)^2d\mu \nonumber
			\\
			&=
\lim_{n\to\infty} \bigg(\sum_{k=1}^n\textbf{v}_{n,k}
\int_\Omega
\bigg(F_{1:n}(A'')+F_{1:n}(\xi^{A''})\bigg)
\bigg(F_{1:n}(A'')+F_{1:n}(\xi^{A''})\bigg)^T
d\mu\, \textbf{v}_{n,k}^T\nonumber
			\\
			&+
			\sum_{k=1}^n\int_\Omega  \bigg((\delta_n(A''))(k)-\sum_{l=1}^n\lambda_{k,l}^{\beta}(\delta_n(A''))(n+l)\bigg)^2 d\mu \nonumber
			\\
			&+
			2\sum_{k=1}^n\int_\Omega  \bigg(B^n_{k,.}(F_{1:n}(A'')+F_{1:n}(\xi^{A''}))- \sum_{l=1}^n\lambda_{k,l}^{\beta}\bigg(B^n_{n+l,.}(F_{1:n}(A'')+F_{1:n}(\xi^{A''}))\bigg) \bigg)\nonumber
			\\
			&\cdot\bigg((\delta_n(A''))(k)-\sum_{l=1}^n\lambda_{k,l}^{\beta}(\delta_n(A''))(n+l)\bigg)d\mu \bigg).
		\end{align}
		The term
		\begin{align*}
			2\sum_{k=1}^n&\int_\Omega  \bigg(B^n_{k,.}(F_{1:n}(A'')+F_{1:n}(\xi^{A''}))- \sum_{l=1}^n\lambda_{k,l}^{\beta}\bigg(B^n_{n+l,.}(F_{1:n}(A'')+F_{1:n}(\xi^{A''}))\bigg) \bigg)
			\\
			&\cdot \bigg((\delta_n(A''))(k)-\sum_{l=1}^n\lambda_{k,l}^{\beta}(\delta_n(A''))(n+l)\bigg)d\mu 
		\end{align*}
		is readily dominated by (using the Cauchy Schwarz-inequality, first for the integral and then for the sum)
		\begin{align*}
			&2\sum_{k=1}^n\left(\int_\Omega  \left(B^n_{k,.}(F_{1:n}(A'')+F_{1:n}(\xi^{A''}))-\sum_{l=1}^n\lambda_{k,l}^{\beta}\left(B^n_{n+l,.}(F_{1:n}(A'')+F_{1:n}(\xi^{A''}))\right) \right)^2 d\mu \right)^{\frac12} 
			\\
			&\cdot\left(\int_\Omega  \left((\delta_n(A''))(k)-\sum_{l=1}^n\lambda_{k,l}^{\beta}(\delta_n(A''))(n+l)\right)^2 d\mu \right)^{\frac12}
			\\
			&\le
			2\left(\sum_{k=1}^n\int_\Omega  \left(B^n_{k,.}(F_{1:n}(A'')+F_{1:n}(\xi^{A''}))- \sum_{l=1}^n\lambda_{k,l}^{\beta}\left(B^n_{n+l,.}(F_{1:n}(A'')+F_{1:n}(\xi^{A''}))\right) \right)^2 d\mu \right)^{\frac12} 
			\\
			&\times\left(\sum_{k=1}^n\int_\Omega  \left((\delta_n(A''))(k)-\sum_{l=1}^n\lambda_{k,l}^{\beta}(\delta_n(A''))(n+l)\right)^2 d\mu \right)^{\frac12}.
		\end{align*}
		This term will converge to zero since, as we will see,  
		\begin{align}\label{deltatermen}
			\sum_{k=1}^n\int_\Omega  \left((\delta_n(A''))(k)-\sum_{l=1}^n\lambda_{k,l}^{\beta}(\delta_n(A''))(n+l)\right)^2 d\mu 
		\end{align}
		converges to zero, while we will show that the term 
		\begin{align}\label{bdd}
			\sum_{k=1}^n\int_\Omega  \left(B^n_{k,.}(F_{1:n}(A'')+F_{1:n}(\xi^{A''}))- \sum_{l=1}^n\lambda_{k,l}^{\beta}\left(B^n_{n+l,.}(F_{1:n}(A'')+F_{1:n}(\xi^{A''}))\right) \right)^2 d\mu ,
		\end{align}
		is bounded. First, we will show that \eqref{deltatermen} converges to zero. Note that $(\delta_n(A''))(k)=F_k((\mathcal{S}_Y-\mathcal{S}_YP_n^{\mathrm{in}})(A''+\xi^{A''}))$ for $1\le k\le n$. By Bessel's inequality and Parseval's identity, 
        
        \begin{align*}
            \sum_{k=1}^n(\delta_n(A''))(k)^2
            &\le 
            \sum_{k=1}^\infty\left(F_k((\mathcal{S}_Y-\mathcal{S}_YP_n^{\mathrm{in}})(A''+\xi^{A''}))\right)^2\\
            &=\lVert(\mathcal{S}_Y-\mathcal{S}_YP_n^{\mathrm{in}})(A''+\xi^{A''})\rVert_{\mathbb{H}_o}^2\\
            &\le
            2\lVert \mathcal{S}_Y(I-P_n^{\mathrm{in}})A''\rVert_{\mathbb{H}_o}^2+ 2\lVert \mathcal{S}_Y(I-P_n^{\mathrm{in}})\xi^{A''}\rVert_{\mathbb{H}_o}^2\\
            &\le
            2\lVert \mathcal{S}\rVert^2\left( \lVert (I-P_n^{\mathrm{in}})A''\rVert_{\mathbb{H}_o^{d}}^2+ \lVert (I-P_n^{\mathrm{in}})\xi^{A''}\rVert_{\mathbb{H}_o^{d}}^2\right),
        \end{align*}
        where the right-most side converges to zero $\mu$-a.e.. Furthermore from the above inequality we also have,
        \begin{align}\label{dbd}
            \sum_{k=1}^n(\delta_n(A''))(k)^2
            &\le
            4\lVert \mathcal{S}\rVert^2\left( \lVert A''\rVert_{\mathbb{H}_o^{d}}^2+ \lVert \xi^{A''}\rVert_{\mathbb{H}_o^{d}}^2\right):=\tilde{M},
        \end{align}
        where $\int_\Omega \tilde{M}d\mu <\infty$. Analogously we also have that 
                \begin{align}\label{dbd2}
            \sum_{k=n+1}^{2n}(\delta_n(A''))(k)^2
            &\le
            4\lVert \mathcal{S}_X\rVert^2\left( \lVert A''\rVert_{\mathbb{H}_o^{d}}^2+ \lVert \xi^{A''}\rVert_{\mathbb{H}_o^{d}}^2\right),
        \end{align}
        It follows from the dominated convergence theorem that
        \begin{align}\label{limdelta}
            \lim_{n\to\infty}\int_\Omega \sum_{k=1}^n(\delta_n(A''))(k)^2d\mu =0.
        \end{align}
Recall that we defined,
\[
S_n^\beta
:=
\sum_{k=1}^n\sum_{\ell=1}^n
\lambda_{k,\ell}^{\beta}\,\phi_k\otimes\psi_\ell
\;\in\;
\mathbb{H}_o\widehat{\otimes}\H
\;\cong\;
\mathcal{HS}(\H^*,\mathbb{H}_o).
\]
We denote by
\[
T_{S_n^\beta}
:\H^*\to\mathbb{H}_o
\]
the Hilbert--Schmidt operator canonically associated with \(S_n^\beta\), defined by
\[
T_{S_n^\beta}(x^*)
:=
\left\langle S_n^\beta,\,x^* \right\rangle_{\H,\H^*}
=
\sum_{k=1}^n\sum_{\ell=1}^n
\lambda_{k,\ell}^{\beta}\,x^*(\psi_\ell)\,\phi_k,
\qquad x^*\in\H^*.
\]
We proceed to bound the latter part of $\delta_n(A'')$, 
\begin{align}\label{deltainl}
\sum_{k=1}^n\left(\sum_{\ell=1}^n\lambda_{k,\ell}^{\beta}\,(\delta_n(A''))(n+\ell)\right)^2
&=
\sum_{k=1}^n\left(\sum_{\ell=1}^n\lambda_{k,\ell}^{\beta}\,
\pairHH{\psi_\ell}{
\mathcal{S}_{X}\!\left(A''+\xi^{A''}\right)
-\mathcal{S}_{X}P_n^{\mathrm{in}}\!\left(A''+\xi^{A''}\right)
}\right)^2\nonumber
\\
&=
\left\lVert
\sum_{k=1}^n\sum_{\ell=1}^n\lambda_{k,\ell}^{\beta}\,
\pairHH{\psi_\ell}{
\mathcal{S}_{X}\!\left(A''+\xi^{A''}\right)
-\mathcal{S}_{X}P_n^{\mathrm{in}}\!\left(A''+\xi^{A''}\right)
}\,
\phi_k
\right\rVert_{\mathbb{H}_o}^2\nonumber
\\
&=
\left\lVert
\left\langle
\sum_{k=1}^n\sum_{\ell=1}^n\lambda_{k,\ell}^{\beta}\,
\phi_k\otimes\psi_\ell,\;
\mathcal{S}_{X}\!\left(A''+\xi^{A''}\right)
-\mathcal{S}_{X}P_n^{\mathrm{in}}\!\left(A''+\xi^{A''}\right)
\right\rangle
\right\rVert_{\mathbb{H}_o}^2\nonumber
\\
&=
\big\lVert
T_{S_n^\beta}\!\big(
\mathcal{S}_{X}\!\left(A''+\xi^{A''}\right)
-\mathcal{S}_{X}P_n^{\mathrm{in}}\!\left(A''+\xi^{A''}\right)
\big)
\big\rVert_{\mathbb{H}_o}^2\nonumber
\\
&\le
\left\lVert
\mathcal{S}_{X}\!\left(A''+\xi^{A''}\right)
-\mathcal{S}_{X}P_n^{\mathrm{in}}\!\left(A''+\xi^{A''}\right)
\right\rVert_{\H^*}^2\,
\lVert T_{S_n^\beta}\rVert_{\mathcal{HS}(\H^*,\mathbb{H}_o)}^2 .
\end{align}

Using the identification
\(\mathcal{HS}(\H^*,\mathbb{H}_o)\cong\mathbb{H}_o\widehat{\otimes}\H\),
orthonormality of \(\{\phi_k\}_{k\ge1}\) and \(\{\psi_\ell\}_{\ell\ge1}\),
and Bessel's inequality, we obtain
\begin{align*}
\lVert T_{S_n^\beta}\rVert_{\mathcal{HS}(\H^*,\mathbb{H}_o)}^2
&=
\sum_{\ell=1}^n\left\lVert
\sum_{k=1}^n \lambda_{k,\ell}^{\beta}\,\phi_k
\right\rVert_{\mathbb{H}_o}^2
=
\sum_{\ell=1}^n\sum_{k=1}^n (\lambda_{k,\ell}^{\beta})^2
\\
&\le
\sum_{\ell=1}^\infty\sum_{k=1}^\infty (\lambda_{k,\ell}^{\beta})^2
=
\lVert\beta\rVert_{\mathcal{HS}(\H^*,\mathbb{H}_o)}^2
=
\lVert T \rVert_{V_1}^2 .
\end{align*}

        \noindent Combining \eqref{deltainl} with
        \begin{align*}
        \left\lVert \mathcal{S}_{X}\left(A''+\xi^{A''}\right)-\mathcal{S}_{X}P_n^{\mathrm{in}}\left(A''+\xi^{A''}\right) \right\rVert_{\H^*}^2
        &=
        \left\lVert \mathcal{S}_{X}\left(A''+\xi^{A''}\right)-\mathcal{S}_{X}P_n^{\mathrm{in}}\left(A''+\xi^{A''}\right)\right\rVert_{\H^*}^2
        \\
        &\le
        \lVert \mathcal{S} \rVert^2 \left(2\lVert(I-P_n^{\mathrm{in}})A''\rVert_{\mathbb{H}_o^{d}}^2 + 2\lVert(I-P_n^{\mathrm{in}})\xi^{A''}\rVert_{\mathbb{H}_o^{d}}^2\right),
        \end{align*}
which converges to zero, implying (together with \eqref{deltainl} and the bound $\lVert T_{S_n^\beta}\rVert_{\mathcal{HS}(\H^*,\mathbb{H}_o)}^2\le \lVert T \rVert_{V_1}^2$) that the left-most side of \eqref{deltainl} converges to zero $\mu$-a.e..
 Moreover
        \begin{align}\label{deltabound}
        \sum_{k=1}^n\left(\sum_{l=1}^n\lambda_{k,l}^{\beta}(\delta_n(A''))(n+l)\right)^2
        \le 
        \lVert  T \rVert_{V_1}^2 \lVert \mathcal{S} \rVert^2 \left(4\lVert A''\rVert_{\mathbb{H}_o^{d}}^2 + 4\lVert \xi^{A''}\rVert_{\mathbb{H}_o^{d}}^2\right)
        \end{align}
        which is $d\mu$-integrable and therefore by the dominated convergence theorem
        \begin{align}\label{limlambdadelta}
            \lim_{n\to\infty}\int_\Omega \sum_{k=1}^n\left(\sum_{l=1}^n\lambda_{k,l}^{\beta}(\delta_n(A''))(n+l)\right)^2d\mu =0
        \end{align}
        Expanding the squares in \eqref{deltatermen} we find
		\begin{align}\label{smalldelta}
			&\sum_{k=1}^n\int_\Omega  \left((\delta_n(A''))(k)-\sum_{l=1}^n\lambda_{k,l}^{\beta}(\delta_n(A''))(n+l)\right)^2 d\mu \nonumber
			\\
			&\le
			2\int_\Omega   \sum_{k=1}^n(\delta_n(A''))(k)^2 d\mu
			+
			4\sum_{k=1}^n\int_\Omega \left(\sum_{l=1}^n\lambda_{k,l}^{\beta}(\delta_n(A''))(n+l) \right)^2d\mu ,
		\end{align}
        which converges to zero due to \eqref{limdelta} and \eqref{limlambdadelta}.
		We shall now establish that \eqref{bdd} is indeed bounded. We have,
		\begin{align*}
			&\sum_{k=1}^n\int_\Omega  \left(B^n_{k,.}(F_{1:n}(A'')+F_{1:n}(\xi^{A''}))- \sum_{l=1}^n\lambda_{k,l}^{\beta}\left(B^n_{n+l,.}(F_{1:n}(A'')+F_{1:n}(\xi^{A''}))\right) \right)^2 d\mu 
			\\
			&\le \sum_{k=1}^n \int_\Omega  2\left(B^n_{k,.}(F_{1:n}(A'')+F_{1:n}(\xi^{A''})) \right)^2 d\mu 
			+
			2\sum_{k=1}^n\int_\Omega \left(\sum_{l=1}^n\lambda_{k,l}^{\beta}\left(B^n_{n+l,.}(F_{1:n}(A'')+F_{1:n}(\xi^{A''}))\right)\right)^2 d\mu  
			\\
			&\le 
			\sum_{k=1}^n4\int_\Omega \left(\left(Z_k^{A''}\right)^2+ (\delta_n(A''))(k)^2\right) d\mu 
			+4\sum_{k=1}^n\int_\Omega \left(\sum_{l=1}^n\lambda_{k,l}^{\beta}\left(\chi^{A''}_l-(\delta_n(A''))(n+l)\right)\right)^2 d\mu  
			\\
			&\le 
			\sum_{k=1}^n4\int_\Omega \left(Z_k^{A''}\right)^2d\mu +4\int_\Omega \sum_{k=1}^n (\delta_n(A'')) (k)^2d\mu 
			+4\sum_{k=1}^n\int_\Omega  \sum_{l=1}^n\left(\lambda_{k,l}^{\beta}\chi^{A''}_l\right)^2 d\mu 
            \\
            &+4\sum_{k=1}^n \sum_{l=1}^n\left(\lambda_{k,l}^{\beta}\right)^2\int_\Omega  (\delta_n(A''))(n+l)^2 d\mu 
            \\
			&\le 4\int_\Omega \lVert Y^{A''}\rVert_{\mathbb{H}_o}^2d\mu 
            +
            32\max\left(\lVert T \rVert_{V_1}^2,1 \right)\lVert \mathcal{S} \rVert^2\int_\Omega  \left(\lVert A''\rVert_{\mathbb{H}_o^{d}}^2 + \lVert \xi^{A''}\rVert_{\mathbb{H}_o^{d}}^2\right)d\mu ,
		\end{align*}
		where we utilized \eqref{SEMA}, \eqref{dbd} the fact that (due to \eqref{dbd2})
        \begin{align*}
        \sum_{k=1}^n \sum_{l=1}^n\left(\lambda_{k,l}^{\beta}\right)^2\int_\Omega  (\delta_n(A''))(n+l)^2 d\mu 
        &\le 
        \sum_{k=1}^n \sum_{l=1}^n\left(\lambda_{k,l}^{\beta}\right)^2 \sum_{r=n+1}^{2n} \int_\Omega(\delta_n(A''))(r)^2 d\mu 
        \\
        &\le 
        \lVert T\rVert_{V_1}^2 \int_\Omega   4\lVert \mathcal{S}\rVert^2\left( \lVert A''\rVert_{\mathbb{H}_o^{d}}^2+ \lVert \xi^{A''}\rVert_{\mathbb{H}_o^{d}}^2\right)  d\mu 
        \end{align*}
        and
        $$\sum_{k=1}^n\int_\Omega  \sum_{l=1}^n\left(\lambda_{k,l}^{\beta}\chi^{A''}_l\right)^2 d\mu 
        \le
        4\lVert T \rVert_{V_1}^2 \lVert S \rVert^2\int_\Omega  \left(\lVert A''\rVert_{\mathbb{H}_o^{d}}^2 + \lVert \xi^{A''}\rVert_{\mathbb{H}_o^{d}}^2\right)d\mu ,$$
        by a calculation analogous to \eqref{deltabound}. This establishes that \eqref{bdd} is indeed bounded.
        Returning to \eqref{R_A} we now have
		for any $A''\in \mathcal{V}$,
\begin{align}\label{R_Afinal}
			R_{A''}(T) 
			&=
			\lim_{n\to\infty}\sum_{k=1}^n \textbf{v}_{n,k}\int_\Omega  \left(F_{1:n}(A'')+F_{1:n}(\xi^{A''})\right)\left(F_{1:n}(A'')+F_{1:n}(\xi^{A''})\right)^T d\mu  \textbf{v}_{n,k}^T.
\end{align}
  \textbf{Step 5: Verify that the cross term vanishes and that the pure baseline component term is invariant.}

Let $A'',A'\in\mathcal V$ be arbitrary and let $\xi^{A''},\xi^{A'}\in\mathcal V$
be the corresponding baseline components. We index coordinates in $\mathbb R^{dn}$ by pairs $(i,k)$ with
$i\in\{1,\ldots, d \}$ and $k\in\{1,\ldots,n\}$.

\emph{Cross term.}
Fix $(i,k)$ and $(j,\ell)$. Let $\iota_r:\mathbb H_o\to\mathbb H_o^{d}$ denote the canonical injection
into the $r$th component, i.e.\ $\iota_r(h)=(0,\ldots,0,h,0,\ldots,0)$. For $u,v\in\mathcal V$ we interpret
$u\otimes v$ as the rank--one operator on $\mathcal V$ defined by
\[
(u\otimes v)(h):=\langle h,v\rangle_{\mathcal V}\,u,\qquad h\in\mathcal V.
\]
Define the operator $\mathcal{I}:\mathcal{V}\to\mathcal{V}$ by
\[
\mathcal{I}u=\Big(\int_\Omega A''\otimes \xi^{A''}\,d\mu\Big)u,
\qquad u\in\mathcal V.
\]
Here $\int_\Omega A''\otimes\xi^{A''},d\mu$ is understood as a Bochner integral in the Hilbert space $\mathcal{HS}(\mathcal V)$.
Then
\begin{align}\label{Ae}
\Big(\int_\Omega F_{1:n}(A'')F_{1:n}(\xi^{A''})^\top\,d\mu\Big)_{(i,k),(j,\ell)}
&=
\int_\Omega \langle A''(i),\phi_k\rangle_{\mathbb H_o}\,
          \langle \xi^{A''}(j),\phi_\ell\rangle_{\mathbb H_o}\,d\mu
\nonumber\\
&=
\int_\Omega \langle A'',\iota_i(\phi_k)\rangle_{\mathcal V}\,
          \langle \xi^{A''},\iota_j(\phi_\ell)\rangle_{\mathcal V}\,d\mu
\nonumber\\
&=
\int_\Omega \Big\langle \big(A''\otimes \xi^{A''}\big)\,\iota_j(\phi_\ell),\;\iota_i(\phi_k)\Big\rangle_{\mathcal V}\,d\mu
\nonumber\\
&=
\Big\langle \mathcal{I}\,\iota_j(\phi_\ell),\;\iota_i(\phi_k)\Big\rangle_{\mathcal V},
\end{align}
above we identified $\iota_i(\phi_k)\in\mathbb H_o^{d}$ with the corresponding constant element of $\mathcal V=L^2(\Omega;\mathbb H_o^{d})$, i.e. $\omega\mapsto \iota_i(\phi_k)$, and kept the same notation.
By assumption $\int_\Omega A''\otimes \xi^{A''}\,d\mu=0$, we have $\mathcal I=0$, and therefore
\[
\int_\Omega F_{1:n}(A'')F_{1:n}(\xi^{A''})^\top\,d\mu = 0
\quad\text{in }\mathbb R^{dn\times dn}.
\]

\medskip

\emph{Pure baseline component term and invariance.}
Fix $(i,k)$ and $(j,\ell)$. By the standing assumption
$\int_\Omega \xi^{A}\otimes \xi^{A}\,d\mu=\Sigma_\xi$, we obtain
\begin{align}\label{obseqvi}
\Big(\int_\Omega F_{1:n}(\xi^{A''})F_{1:n}(\xi^{A''})^\top\,d\mu\Big)_{(i,k),(j,\ell)}
&=
\int_\Omega
\langle \xi^{A''}(i),\phi_k\rangle_{\mathbb H_o}\,
\langle \xi^{A''}(j),\phi_\ell\rangle_{\mathbb H_o}\,d\mu
\nonumber\\
&=
\int_\Omega
\Big\langle \big(\xi^{A''}\otimes \xi^{A''}\big)\,\iota_j(\phi_\ell),\;
\iota_i(\phi_k)\Big\rangle_{\mathcal V}\,d\mu
\nonumber\\
&=
\Big\langle
\Big(\int_\Omega \xi^{A''}\otimes \xi^{A''}\,d\mu\Big)\,\iota_j(\phi_\ell),\;
\iota_i(\phi_k)
\Big\rangle_{\mathcal V}.
\end{align}
By the standing assumption
\[
\int_\Omega \xi^{A}\otimes \xi^{A}\,d\mu=\Sigma_\xi
\qquad \text{for all admissible }A,
\]
the right hand side depends only on $\Sigma_\xi$ and is therefore independent of the admissible source.
Consequently,
\[
\int_\Omega F_{1:n}(\xi^{A''})F_{1:n}(\xi^{A''})^\top\,d\mu
=
\int_\Omega F_{1:n}(\xi^{A'})F_{1:n}(\xi^{A'})^\top\,d\mu.
\]

\medskip

\noindent\textbf{Step 6: Source-continuity in the cost}
        \\
       Let $A',A''\in\mathcal{V}$ be arbitrary. Recall \eqref{R_Afinal},
        \begin{align*}
			R_{A^*}(T)
			&=
			\lim_{n\to\infty} \sum_{k=1}^n \textbf{v}_{n,k}\int_\Omega  \left(F_{1:n}(A^*)+F_{1:n}(\xi^{A^*})\right)\left(F_{1:n}(A^*)+F_{1:n}(\xi^{A^*})\right)^T d\mu  \textbf{v}_{n,k}^T,
        \end{align*}
        for any $A^*\in\mathcal{V}$
        Setting $A^*=0$ above gives us,
        \begin{align}
        \lim_{n\to\infty} \sum_{k=1}^n \textbf{v}_{n,k}\int_\Omega  F_{1:n}(\xi^{0})F_{1:n}(\xi^{0})^T d\mu  \textbf{v}_{n,k}^T     
            &=
			\int_\Omega \left\lVert
            S_1(\xi^{0}) -  
            T\left(\mathcal{S}_X\xi^{0}\right)
        \right\rVert_{\mathbb{H}_o}^2 d\mu ,
        \end{align}
        which is finite. 
        By the standing assumption $\int_\Omega \xi^{A}\otimes \xi^{A}\,d\mu=\Sigma_\xi$, we obtain
\[
\Big(\int_\Omega F_{1:n}(\xi^{A''})F_{1:n}(\xi^{A''})^\top\,d\mu\Big)_{(i,k),(j,\ell)}
=
\big\langle \Sigma_\xi\,\iota_j(\phi_\ell),\;\iota_i(\phi_k)\big\rangle_{\mathcal V},
\]
which is independent of the admissible source. Combining this with \eqref{obseqvi} gives us,

        \begin{align}
        \sum_{k=1}^n \textbf{v}_{n,k}\int_\Omega  F_{1:n}(\xi^{0})F_{1:n}(\xi^{0})^T d\mu  \textbf{v}_{n,k}^T     
            &=
			\sum_{k=1}^n \textbf{v}_{n,k}\int_\Omega  F_{1:n}(\xi^{A''})F_{1:n}(\xi^{A''})^T d\mu  \textbf{v}_{n,k}^T\\
            &=
            \sum_{k=1}^n \textbf{v}_{n,k}\int_\Omega  F_{1:n}(\xi^{A'})F_{1:n}(\xi^{A'})^T d\mu  \textbf{v}_{n,k}^T 
        \end{align}
        which implies that all their limits are equal and finite. Setting $\xi=0$ gives us
        \begin{align}
        \lim_{n\to\infty} \sum_{k=1}^n \textbf{v}_{n,k}\int_\Omega  F_{1:n}(A'')F_{1:n}(A'')^T d\mu  \textbf{v}_{n,k}^T:=h_{A''}(T)     
            &=
			\int_\Omega \left\lVert
            \mathcal{S}_Y(A'') -  
            T\left(S_{2}A''\right)
       \right\rVert_{\mathbb{H}_o}^2 d\mu .
        \end{align}
        Which is also finite. Due to \eqref{Ae} the limit of
        $ \textbf{v}_{n,k}\int_\Omega  F_{1:n}(A'')F_{1:n}(\xi^{A''})^T d\mu  \textbf{v}_{n,k}^T$
        exists and is zero. Having established that the above individual limits exist, we may therefore split   
        \begin{align*}
			R_{A''}(T)  -R_{A'}(T) 
			&=
			\lim_{n\to\infty} \sum_{k=1}^n \textbf{v}_{n,k}\int_\Omega  \left(F_{1:n}(A'')+F_{1:n}(\xi^{A''})\right)\left(F_{1:n}(A'')+F_{1:n}(\xi^{A''})\right)^T d\mu  \textbf{v}_{n,k}^T
            \\
            &-
            \lim_{n\to\infty} \sum_{k=1}^n \textbf{v}_{n,k}\int_\Omega  \left(F_{1:n}(A')+F_{1:n}(\xi^{A'})\right)\left(F_{1:n}(A')+F_{1:n}(\xi^{A'})\right)^T d\mu  \textbf{v}_{n,k}^T
            \\
            &=
            \lim_{n\to\infty} \sum_{k=1}^n \textbf{v}_{n,k}\int_\Omega  F_{1:n}(A'')F_{1:n}(A'')^T d\mu  \textbf{v}_{n,k}^T
            -
            \lim_{n\to\infty} \sum_{k=1}^n \textbf{v}_{n,k}\int_\Omega  F_{1:n}(A')F_{1:n}(A')^T d\mu  \textbf{v}_{n,k}^T
            \\
            &+2\lim_{n\to\infty} \sum_{k=1}^n \textbf{v}_{n,k}\int_\Omega  F_{1:n}(A'')F_{1:n}(\xi^{A''})^T d\mu  \textbf{v}_{n,k}^T
            -2\lim_{n\to\infty} \sum_{k=1}^n \textbf{v}_{n,k}\int_\Omega  F_{1:n}(A')F_{1:n}(\xi^{A'})^T d\mu  \textbf{v}_{n,k}^T
            \\
            &+\lim_{n\to\infty} \sum_{k=1}^n \textbf{v}_{n,k}\int_\Omega  F_{1:n}(\xi^{A''})F_{1:n}(\xi^{A''})^T d\mu  \textbf{v}_{n,k}^T
            -\lim_{n\to\infty} \sum_{k=1}^n \textbf{v}_{n,k}\int_\Omega  F_{1:n}(\xi^{A'})F_{1:n}(\xi^{A'})^T d\mu  \textbf{v}_{n,k}^T
            \\
            &=
             \lim_{n\to\infty} \sum_{k=1}^n \textbf{v}_{n,k}\int_\Omega  F_{1:n}(A'')F_{1:n}(A'')^T d\mu  \textbf{v}_{n,k}^T
            -\lim_{n\to\infty} \sum_{k=1}^n \textbf{v}_{n,k}\int_\Omega  F_{1:n}(A')F_{1:n}(A')^T d\mu  \textbf{v}_{n,k}^T.
        \end{align*}
        It follows that
        \begin{align*}
        R_{A''}(T)-R_{A'}(T)=h_{A''}(T)-h_{A'}(T).
        \end{align*}
        Therefore,
        		\begin{align}\label{RDelta}
			\left| R_{A'}(T)-R_{A''}(T)\right|
			&\le
			\left|\int_\Omega\left( \left\lVert \mathcal{S}_Y(A')\right\rVert_{\mathbb{H}_o}^2
			-
			\left\lVert \mathcal{S}_Y(A'')\right\rVert_{\mathbb{H}_o}^2\right)d\mu \right|\nonumber
			\\
			&+
			2\left|\int_\Omega\left( \left\langle \mathcal{S}_Y(A') ,T\left(\mathcal{S}_{X}A'\right)\right\rangle_{\mathbb{H}_o}
			-
			\left\langle \mathcal{S}_Y(A'') ,T\left(\mathcal{S}_{X}A''\right)\right\rangle_{\mathbb{H}_o}\right)d\mu 
			\right|\nonumber
			\\
			&+
			\left|\int_\Omega \left(\left\lVert T\left(S_{X}A'\right)\right\rVert_{\mathbb{H}_o}^2 -
			\left\lVert T\left(\mathcal{S}_XA''\right)\right\rVert_{\mathbb{H}_o}^2\right)d\mu 
			\right|.
		\end{align}
        For the first term on the right-hand side of \eqref{RDelta}
        \begin{align*}
        \left|\int_\Omega\left( \left\lVert \mathcal{S}_Y(A')\right\rVert_{\mathbb{H}_o}^2
			-
			\left\lVert \mathcal{S}_Y(A'')\right\rVert_{\mathbb{H}_o}^2\right)d\mu \right|
            &\le
            \int_\Omega \left|\left\lVert \mathcal{S}_Y(A')\right\rVert_{\mathbb{H}_o}
			+
			\left\lVert \mathcal{S}_Y(A'')\right\rVert_{\mathbb{H}_o}\right|\left|\left\lVert \mathcal{S}_Y(A')\right\rVert_{\mathbb{H}_o}
			-
			\left\lVert \mathcal{S}_Y(A'')\right\rVert_{\mathbb{H}_o}\right|
			d\mu 
            \\
            &\le
            \int_\Omega \left|\left\lVert \mathcal{S}_Y(A')\right\rVert_{\mathbb{H}_o}
			+
			\left\lVert \mathcal{S}_Y(A'')\right\rVert_{\mathbb{H}_o}\right|\left|\left\lVert \mathcal{S}_Y(A')-\mathcal{S}_Y(A'')\right\rVert_{\mathbb{H}_o}
			\right|
			d\mu
            \\
            &=
            \int_\Omega \left|\left\lVert \mathcal{S}_Y(A')\right\rVert_{\mathbb{H}_o}
			+
			\left\lVert \mathcal{S}_Y(A'')\right\rVert_{\mathbb{H}_o}\right|\left|\left\lVert \mathcal{S}_Y(A'-A'')\right\rVert_{\mathbb{H}_o}
			\right|
			d\mu
            \\
            &\le
            \left(\int_\Omega \lVert \mathcal{S}_Y(A'-A'')\rVert_{\mathbb{H}_o}^2d\mu \right)^{\frac12}
            \left(\int_\Omega \left(2\left\lVert \mathcal{S}_Y(A')\right\rVert_{\mathbb{H}_o}^2
			+
			2\left\lVert \mathcal{S}_Y(A'')\right\rVert_{\mathbb{H}_o}^2\right)d\mu \right)^{\frac12}
            \\
            &\le
            \sqrt{2}\lVert \mathcal{S}\rVert^2 \lVert A'-A''\rVert_{\mathcal{V}}\left(\lVert A''\rVert_{\mathcal{V}} + \lVert A'\rVert_{\mathcal{V}}\right),
        \end{align*}
        where we used Cauchy-Schwarz and the reverse triangle inequality and the fact that $\lVert \mathcal{S}_Y \rVert \le \lVert \mathcal{S} \rVert$.
        For the second term we also utilize these inequalities
        \begin{align*}
        &\left|\int_\Omega\left( \left\langle \mathcal{S}_Y(A') ,T\left(\mathcal{S}_{X}A'\right)\right\rangle_{\mathbb{H}_o}
			-
			\left\langle \mathcal{S}_Y(A'') ,T\left(\mathcal{S}_{X}A''\right)\right\rangle_{\mathbb{H}_o}\right)d\mu 
			\right|
            \\
            &=
            \left|\int_\Omega 
             \left\langle \mathcal{S}_Y(A'-A''),T\left(\mathcal{S}_{X}A'\right)\right\rangle_{\mathbb{H}_o}
             +
             \left\langle \mathcal{S}_Y(A''),T\left(\mathcal{S}_{X}(A''-A')\right)\right\rangle_{\mathbb{H}_o}
           d\mu\right|
            \\
            &\le
            \int_\Omega \left|
             \left\langle \mathcal{S}_Y(A'-A''),T\left(\mathcal{S}_{X}A'\right)\right\rangle_{\mathbb{H}_o}\right|d\mu 
             +
             \int_\Omega \left|\left\langle \mathcal{S}_Y(A'') ,T\left(\mathcal{S}_{X}(A''-A')\right)\right\rangle_{\mathbb{H}_o}\right|
           d\mu 		
            \\
            &\le
            \int_\Omega  \lVert \mathcal{S}_Y(A'-A'')\rVert_{\mathbb{H}_o}\lVert  T\left(\mathcal{S}_{X}A'\right)\rVert_{\mathbb{H}_o}d\mu 
+
            \int_\Omega  \lVert \mathcal{S}_Y(A'')\rVert_{\mathbb{H}_o}\lVert  T\left(\mathcal{S}_{X}(A''-A')\right)\rVert_{\mathbb{H}_o}d\mu 
            \\
            &\le
             \lVert \mathcal{S}\rVert \int_\Omega  \lVert A'-A''\rVert_{\mathbb{H}_o}\lVert  \lVert  T \rVert_{V_1}\lVert \mathcal{S}_{X}A'\rVert_{H^*}d\mu 
             +
             \lVert \mathcal{S}\rVert \int_\Omega  \lVert A''\rVert_{\mathbb{H}_o}\lVert  \lVert  T \rVert_{V_1}\lVert \mathcal{S}_X(A''-A')\rVert_{H^*}d\mu 
             \\
             &\le
             \lVert \mathcal{S}\rVert^2\lVert T \rVert_{V_1}\lVert A'-A''\rVert_{\mathcal{V}} \left(\lVert A''\rVert_{\mathcal{V}} + \lVert A'\rVert_{\mathcal{V}}\right).
        \end{align*}
        For the third term, we proceed analogously to the first term,
        \begin{align*}
        &\left|\int_\Omega \left(\left\lVert T\left(\mathcal{S}_XA'\right)\right\rVert_{\mathbb{H}_o}^2 -
			\left\lVert T\left(\mathcal{S}_XA''\right)\right\rVert_{\mathbb{H}_o}^2\right)d\mu 
			\right|
            \\
           &\le
            \int_\Omega \left|\left\lVert T\left(\mathcal{S}_XA'\right)\right\rVert_{\mathbb{H}_o} -
			\left\lVert T\left(\mathcal{S}_XA''\right)\right\rVert_{\mathbb{H}_o}\right|\left|\left\lVert T\left(\mathcal{S}_XA'\right)\right\rVert_{\mathbb{H}_o} +
			\left\lVert T\left(\mathcal{S}_XA''\right)\right\rVert_{\mathbb{H}_o}\right|d\mu 
            \\
            &\le
            \int_\Omega \left\lVert T\left(\mathcal{S}_X(A'-A'')\right)\right\rVert_{\mathbb{H}_o}\left(\left\lVert T\left(\mathcal{S}_XA'\right)\right\rVert_{\mathbb{H}_o} +
			\left\lVert T\left(\mathcal{S}_XA''\right)\right\rVert_{\mathbb{H}_o}\right)d\mu
            \\
            &\le
            \left(\int_\Omega \left\lVert T\left(\mathcal{S}_X(A'-A'')\right)\right\rVert_{\mathbb{H}_o}^2d\mu \right)^{\frac12}
            \left(\int_\Omega \left(2\left\lVert T\left(\mathcal{S}_XA'\right)\right\rVert_{\mathbb{H}_o}^2 +
			2\left\lVert T\left(\mathcal{S}_XA''\right)\right\rVert_{\mathbb{H}_o}^2\right)d\mu \right)^{\frac12} 
            \\
            &\le
            \lVert \mathcal{S}\rVert^2\lVert  T \rVert_{V_1}^2\lVert A'-A''\rVert_{\mathcal{V}} \left(\lVert A''\rVert_{\mathcal{V}} + \lVert A'\rVert_{\mathcal{V}}\right).
        \end{align*}
        This allows us to conclude that
        \begin{align*}
        \left| R_{A'}T-R_{A''}(T)\right|
			&\le  (2+\sqrt{2})\lVert \mathcal{S}\rVert^2 \max\left(1,\lVert T \rVert_{V_1}^2\right)\lVert A'-A''\rVert_{\mathcal{V}} \left(\lVert A''\rVert_{\mathcal{V}} + \lVert A'\rVert_{\mathcal{V}}\right)  
        \end{align*}
        and therefore if $A\in\mathcal{V}$ and $\{A_n\}_{n\in\N}\subset\mathcal{V}$ are such that $A_n\xrightarrow{\mathcal{V}}A$ then
        \begin{align}\label{ContSource}
            \lim_{n\to\infty}R_{A_n}(T)=R_{A}(T), \forall T\in V_1
        \end{align}

		\textbf{Step 7: Optimize over the sources}
		\\
				Recall the definition of $\textbf{v}_n$ from step 4. We utilize \eqref{Ae} and \eqref{obseqvi} when we now return to \eqref{R_Afinal},
		\begin{align*}
			R_{A''}(T)
			&=
			\lim_{n\to\infty}\sum_{k=1}^n\textbf{v}_{n,k}\int_\Omega \left(F_{1:n}(A'') +F_{1:n}(\xi^{A''}) \right)\left(F_{1:n}(A'') +F_{1:n}(\xi^{A''})\right)^T d\mu \textbf{v}_{n,k}^T
			\\
			&=
			\lim_{n\to\infty}\sum_{k=1}^n\textbf{v}_{n,k}\int_\Omega F_{1:n}(A'')F_{1:n}(A'')^T d\mu\textbf{v}_{n,k}^T+\lim_{n\to\infty}\sum_{k=1}^n\textbf{v}_{n,k}\int_\Omega  F_{1:n}(\xi^{A''})F_{1:n}(\xi^{A''})^T d\mu\textbf{v}_{n,k}^T
			\\
			&=\lim_{n\to\infty}\sum_{k=1}^n\textbf{v}_{n,k}\int_\Omega F_{1:n}(A'')F_{1:n}(A'')^T d\mu \textbf{v}_{n,k}^T+\lim_{n\to\infty}\sum_{k=1}^n\textbf{v}_{n,k}\int_\Omega  F_{1:n}(\xi^{A})F_{1:n}(\xi^{A})^T d\mu \textbf{v}_{n,k}^T
		\end{align*}
		and similarly we have
		\begin{align}\label{RAformula}
			R_{A}(T)
			=\lim_{n\to\infty}\sum_{k=1}^n\textbf{v}_{n,k}\int_\Omega  F_{1:n}(A)F_{1:n}(A)^T d\mu \textbf{v}_{n,k}^T+\lim_{n\to\infty}\sum_{k=1}^n\textbf{v}_{n,k}\int_\Omega  F_{1:n}(\xi^{A})F_{1:n}(\xi^{A})^T d\mu \textbf{v}_{n,k}^T.
		\end{align}
		Take a sequence $\{A_n\}_{n\in\N}\subset C_{\mathcal{A}}(A)$ such that $\lim_{n\to\infty}R_{A_n}(T)=\sup_{A'\in C_{\mathcal{A}}(A)}R_{A'}(T)$. Fix any $\Delta>0$. Let $\tilde{A}_\Delta\in C_{\mathcal{A}}(A)$ be such that $\lVert \tilde{A}_\Delta-A\rVert_{\mathcal{V}}<\eta$ where $\eta$ is chosen such that $\left|R_{\tilde{A}_\Delta}(T)-R_{A}(T)\right|<\Delta$, which is possible due to \eqref{ContSource} and the fact that $A\in\bar{\mathcal{A}}$. Define the sets
		$$C_m=\{\tilde{A}_\Delta\}\cup\left(\bigcup_{k=1}^m\{A_k\}\right), m\in\N.$$
		Fix $m\in\N$. Since there are only finitely many elements in $C_m$, we have that 
        $$\sum_{k=1}^n\textbf{v}_{n,k}\int_\Omega F_{1:n}(A'')F_{1:n}(A'')^T d\mu\textbf{v}_{n,k}^T$$ 
        converges uniformly over all $A''\in C_m$ as $n\to\infty$, so we may take $N\in\N$ such that 
		$$\left|\lim_{n\to\infty}\sum_{k=1}^n\textbf{v}_{n,k}\int_\Omega  F_{1:n}(A'')F_{1:n}(A'')^T d\mu\textbf{v}_{n,k}^T -\sum_{k=1}^N\textbf{v}_{N,k}\int_\Omega  F_{1:N}(A'')F_{1:N}(A'')^T d\mu\textbf{v}_{N,k}^T\right|<\Delta, \forall A''\in C_m$$
		and
		$$\left|\lim_{n\to\infty}\sum_{k=1}^n\textbf{v}_{n,k}\int_\Omega  F_{1:n}(A)F_{1:n}(A)^T d\mu\textbf{v}_{n,k}^T -\sum_{k=1}^N\textbf{v}_{N,k}\int_\Omega  F_{1:N}(A)F_{1:N}(A)^T d\mu \textbf{v}_{N,k}^T\right|<\Delta.$$
		For each $1\le r\le N$, let
\[
g_i^{(r)}(s):=\sum_{m=1}^{N} v_{N,r}((i-1)N+m)\,\phi_m(s),
\qquad 1\le i\le  d .
\]
Set $g^{(r)}:=(g_1^{(r)},\ldots,g_{d}^{(r)})$. Then clearly $g_i^{(r)}\in \mathbb{H}_o$. 
Fix $N\in\mathbb N$ and $1\le k\le N$. Define, for $1\le i\le  d $,
\[
g_i^{(N,k)}:=\sum_{m=1}^N \mathbf v_{N,k}((i-1)N+m)\,\phi_m\in\mathbb H_o,
\qquad
g^{(N,k)}:=(g_1^{(N,k)},\ldots,g_{d}^{(N,k)})\in\mathbb H_o^{d}.
\]
Let $P_N:\mathbb H_o\to \mathrm{span}\{\phi_1,\ldots,\phi_N\}$ be the orthogonal projection
(applied componentwise on $\mathcal V$). Recall,
\[
\big(F_{1:N}(u)\big)_{(i,m)}=\langle u(i),\phi_m\rangle_{\mathbb H_o},
\qquad 1\le i\le  d ,\ 1\le m\le N,
\]
so that $F_{1:N}(g^{(N,k)})=\mathbf v_{N,k}$.
Define the finite-rank operator $\Sigma_{A''}^{(N)}:\mathcal H\to\mathcal H$ by
\[
\Sigma_{A''}^{(N)}:=\int_\Omega (P_N A'')\otimes (P_N A'')\,d\mu,
\qquad
(u\otimes v)(h):=\langle h,v\rangle_{\mathcal V}\,u.
\]
Define the pointwise Hilbert product space
\[
\mathcal H:=\mathbb H_o^{d},
\qquad
\langle u,v\rangle_{\mathcal H}:=\sum_{i=1}^{d}\langle u(i),v(i)\rangle_{\mathbb H_o},
\qquad u,v\in\mathcal H .
\]
For $n\in\mathbb N$, let $P_n:\mathbb H_o\to \mathrm{span}\{\phi_1,\ldots,\phi_n\}$ denote the orthogonal projection
(applied componentwise on $\mathcal H$). Also recall that, for $u\in\mathcal H$,
\[
\big(F_{1:n}(u)\big)_{(i,m)}:=\langle u(i),\phi_m\rangle_{\mathbb H_o},
\qquad 1\le i\le  d ,\ 1\le m\le n,
\]
so that $F_{1:n}(u)\in\mathbb R^{ dn}$. Fix $n\in\mathbb N$ and $1\le k\le n$. Define, for $1\le i\le  d $,
\[
g_i^{(n,k)}:=\sum_{m=1}^n \mathbf v_{n,k}((i-1)n+m)\,\phi_m\in\mathbb H_o,
\qquad
g^{(n,k)}:=(g_1^{(n,k)},\ldots,g_{d}^{(n,k)})\in\mathcal H.
\]
By construction,
$
F_{1:n}\big(g^{(n,k)}\big)=\mathbf v_{n,k}
$. Therefore, 
\begin{align}\label{truncen}
&\int_\Omega \big\langle g^{(n,k)},P_nA''\big\rangle_{\mathcal H}^2\,d\mu
\nonumber\\
&=
\int_\Omega 
\Big(\sum_{i=1}^{d}\big\langle g_i^{(n,k)},P_nA''(i)\big\rangle_{\mathbb H_o}\Big)
\Big(\sum_{j=1}^{d}\big\langle g_j^{(n,k)},P_nA''(j)\big\rangle_{\mathbb H_o}\Big)
\,d\mu
\nonumber\\
&=
\sum_{i=1}^{d}\sum_{j=1}^{d}
\int_\Omega 
\big\langle g_i^{(n,k)},P_nA''(i)\big\rangle_{\mathbb H_o}\,
\big\langle g_j^{(n,k)},P_nA''(j)\big\rangle_{\mathbb H_o}\,
d\mu
\nonumber\\
&=
\sum_{i=1}^{d}\sum_{j=1}^{d}
\int_\Omega 
\Big(\sum_{m=1}^{n}\mathbf v_{n,k}((i-1)n+m)\,\langle \phi_m,P_nA''(i)\rangle_{\mathbb H_o}\Big)
\Big(\sum_{\ell=1}^{n}\mathbf v_{n,k}((j-1)n+\ell)\,\langle \phi_\ell,P_nA''(j)\rangle_{\mathbb H_o}\Big)
d\mu
\nonumber\\
&=
\sum_{i=1}^{d}\sum_{j=1}^{d}\sum_{m=1}^{n}\sum_{\ell=1}^{n}
\mathbf v_{n,k}((i-1)n+m)\,
\Big(\int_\Omega 
\langle \phi_m,P_nA''(i)\rangle_{\mathbb H_o}\,
\langle \phi_\ell,P_nA''(j)\rangle_{\mathbb H_o}\,
d\mu\Big)\,
\mathbf v_{n,k}((j-1)n+\ell)
\nonumber\\
&=
\sum_{i=1}^{d}\sum_{j=1}^{d}\sum_{m=1}^{n}\sum_{\ell=1}^{n}
\mathbf v_{n,k}((i-1)n+m)\,
\Big(\int_\Omega 
\langle A''(i),\phi_m\rangle_{\mathbb H_o}\,
\langle A''(j),\phi_\ell\rangle_{\mathbb H_o}\,
d\mu\Big)\,
\mathbf v_{n,k}((j-1)n+\ell)
\nonumber\\
&=
\mathbf v_{n,k}\Big(\int_\Omega F_{1:n}(A'')\,F_{1:n}(A'')^\top\,d\mu\Big)\mathbf v_{n,k}^\top .
\end{align}

Consequently, summing over $k$ yields
\begin{align}\label{truncen-sum}
\sum_{k=1}^N \big\langle g^{(N,k)},\Sigma_{A''}^{(N)}g^{(N,k)}\big\rangle_{\mathcal H}
=
\sum_{k=1}^N
\mathbf v_{N,k}\Big(\int_\Omega F_{1:N}(A'')\,F_{1:N}(A'')^\top\,d\mu\Big)\mathbf v_{N,k}^\top.
\end{align}
By the definition of $C_{\mathcal A}(A)$, step 2 (used to pass between the compressed and regular covariance operators) \eqref{truncen}, and since
$A''\in C_m\subset C_{\mathcal A}(A)$, we obtain for each $1\le k\le N$,

\small
\begin{align*}
\sum_{k=1}^N\mathbf v_{N,k}\Big(\int_\Omega F_{1:N}(A'')F_{1:N}(A'')^\top\,d\mu\Big)\mathbf v_{N,k}^\top
&=
\sum_{k=1}^N\big\langle g^{(N,k)},\Sigma_{A''}^{(N)}g^{(N,k)}\big\rangle_{\mathcal H}
\\
&=
\sum_{k=1}^N\big\langle g^{(N,k)},\Sigma_{A''}g^{(N,k)}\big\rangle_{\mathcal H}
\\
&\le
\sum_{k=1}^N\,\big\langle g^{(N,k)},\Sigma_A g^{(N,k)}\big\rangle_{\mathcal H}
\\
&=
\,\sum_{k=1}^N\big\langle g^{(N,k)},\Sigma_A^{(N)} g^{(N,k)}\big\rangle_{\mathcal H}
\\
&=
 \sum_{k=1}^N\,\mathbf v_{N,k}\Big(\int_\Omega F_{1:N}(A)F_{1:N}(A)^\top\,d\mu\Big)\mathbf v_{N,k}^\top
\\
&\le
\lim_{n\to\infty} \sum_{k=1}^N\mathbf v_{n,k}\Big(\int_\Omega F_{1:n}(A)F_{1:n}(A)^\top\,d\mu\Big)\mathbf v_{n,k}^\top+\Delta.
\end{align*}
\normalsize
		Applying \eqref{Ae}  and\eqref{obseqvi} yields
		\begin{align*}
			&\lim_{n\to\infty} \sum_{k=1}^n\textbf{v}_{n,k}\int_\Omega  \left(F_{1:n}(A) +F_{1:n}(\xi^A) \right)\left(F_{1:n}(A) +F_{1:n}(\xi^A)\right)^T d\mu\textbf{v}_{n,k}^T
			\\
            =
            &\lim_{n\to\infty} \sum_{k=1}^n\textbf{v}_{n,k}\int_\Omega  \left(F_{1:n}(A) +F_{1:n}(\xi^{A}) \right)\left(F_{1:n}(A) +F_{1:n}(\xi^{A})\right)^T d\mu \textbf{v}_{n,k}^T
            =
            R_{A}(T).
		\end{align*}
		It therefore follows that
		$
			R_{A''}(T)
			\le
			 R_{A}(T)  +\Delta,
		$
		for all $A''\in C_m$. Since $\tilde{A}_\Delta\in C_m$, for every $m\in\N$, we have
		\begin{align*}
			\max_{A''\in C_m}R_{A''}(T)&\ge  R_{\tilde{A}_\Delta}(T)
			\ge R_{A}(T)-\Delta.
		\end{align*}
		Hence 
		\begin{align*}
			\left|\max_{A''\in C_m}R_{A''}(T)-R_{A}\right|<\Delta.
		\end{align*}
		Since $\lim_{n\to\infty}R_{A_n}(T)=\sup_{A'\in C_{\mathcal{A}}(A)}R_{A'}(T)$, and $R_{A_n}(T)\le \sup_{A'\in C_{\mathcal{A}}(A)}R_{A'}(T)$ (since $A_n\in C_{\mathcal{A}}(A)$) it follows that $\lim_{m\to\infty}\max_{A''\in C_m}R_{A''}(T) =\sup_{A'\in C_{\mathcal{A}}(A)}R_{A'}(T)$ and therefore there exists $M\in\N$ such that if $m\ge M$, $\left| \max_{A''\in C_m}R_{A''}(T) -\sup_{A'\in C_{\mathcal{A}}(A)}R_{A'}(T)\right|<\Delta$.
		Therefore, for $m\ge M$
		\begin{align*}
			\left| \sup_{A'\in C_{\mathcal{A}}(A)}R_{A'}(T)-R_{A}(T)\right|
			&\le 
			\left| \sup_{A'\in  C_{\mathcal{A}}(A)}R_{A'}(T)-\max_{A''\in C_m}R_{A''}(T)\right|
			\\
			&+
			\left|\max_{A''\in C_m}R_{A''}(T)-R_{A}(T)\right|< 2\Delta
		\end{align*}
		and by letting $\Delta\to 0$ we get
		$$ \sup_{A'\in C_{\mathcal{A}}(A)} R_{A'}(T)=R_{A}(T),$$
		as was to be shown.
\end{proof}
	
\subsection{Proof of Theorem \ref{minim}}

\begin{proof}
By Theorem~\ref{WR1}, 
$
\sup_{A'\in C_{\mathcal{A}}(A)} R_{A'}(T)=R_{A}(T).
$
Let $g:\mathcal{H}\to\R$ (recall $\mathcal{H}=L^2(\Omega;\mathbb{H}_o)$) be $g(u):=\lVert u-Y^{A}\rVert_{\mathcal{H}}^{2}$.  
Its Fr\'echet derivative is
\begin{equation}\label{eq:g-deriv}
g'(u)[v]\ =\ 2\,\langle u-Y^{A},v\rangle_{\mathcal H}
\ =\ 2\,J_{\mathcal{H}}(u-Y^{A})(v),
\qquad u,v\in\mathcal{H}.
\end{equation}
Since $R_{A}=g\circ\Gamma$ and $\Gamma$ is bounded linear, the Banach chain rule gives, for $T,h\in V_1$,
\begin{align}
DR_{A}(T)[h]
&= g'(\Gamma T)(\Gamma h) \nonumber \\
&= 2\,J_{\mathcal{H}}(\Gamma T-Y^{A})(\Gamma h) = 2\,(\Gamma^{\ast}J_{\mathcal{H}}(\Gamma T-Y^{A}))(h).
\end{align}
Thus $DR_{A}(T)=2\,\Gamma^{\ast}J_{\mathcal{H}}(\Gamma T-Y^{A})\in (V_1)^{\ast}$.

The necessary first-order condition yields $DR_{A}( T^{\ast})=0$, i.e.
$
\Gamma^{\ast}J_{\mathcal{H}}(\Gamma T^{\ast}-Y^{A})=0,
$
equivalently $C_{XX} T^{\ast}=C_{XY}$ in $(V_1)^{\ast}$.  
Conversely, if $C_{XX} T^{\ast}=C_{XY}$ then for any $h\in V_1$,
\begin{align}
R_{A}( T^{\ast}+h)-R_{A}( T^{\ast})
= 2\,J_{\mathcal{H}}(\Gamma T^{\ast}-Y^{A})(\Gamma h)
   +\lVert\Gamma h\rVert_{\mathcal{H}}^{2} 
= \lVert\Gamma h\rVert_{\mathcal{H}}^{2}\ \ge 0,
\end{align}
so $ T^{\ast}$ minimizes $R_{A}$.  
This proves (i).

\noindent For (ii), suppose $ T_1, T_2$ both solve $C_{XX} T=C_{XY}$.  
Then
$
C_{XX}( T_1- T_2)=0,
$
so $ T_1- T_2\in\Ker(C_{XX})$.  
Thus the full solution set is $ T_0+\Ker(C_{XX})$.
Moreover, since $C_{XX}=\Gamma^{\ast}J_{\mathcal H}\Gamma$, we have $\ker(C_{XX})=\ker(\Gamma)$.
Hence the minimizer is unique if and only if $\Ker(\Gamma)=\{0\}$.
\end{proof}

\subsection{Proof of Corollary \ref{cor:coord-min}}
\begin{proof}[Proof of Corollary~\ref{cor:coord-min} ]
We start with proving (i). By Theorem~\ref{minim}, 
$$\min_{ T\in V_1}\sup_{A'\in C_{\mathcal A}(A)} R_{A'}(T)
=\min_{ T\in V_1} R_{A}(T)$$ 
and the minimizer satisfies 
$C_{XX} T^*=C_{XY}$.  Note that
$Q_{A}( T, T)=\langle v,\Sigma v\rangle_{\ell^2},$
where $ T=\sum_{k,l\ge 1} v_{k,l}\,\phi_k\psi_l$ and therefore $\lVert T \rVert_{V_1}=\lVert v\rVert_{\ell^2}$. By assumption $Q_{A}( T, T)\ge c \lVert T \rVert_{V_1}^2$ which is equivalent to
$\langle v,\Sigma v\rangle_{\ell^2}\ge c \lVert v\rVert_{\ell^2}^2,$
i.e. $C_{XX}\succeq cI$. Recall that $V_1$ is a Hilbert space. Define the continuous bilinear form

$$a( T,\eta):=\mathcal Q_{A}( T,\eta),\qquad  T,\eta\in V_1,$$
and the continuous linear functional

$$f(\eta):=\mathcal L_{A}(\eta),\qquad \eta\in V_1.$$
By hypothesis there exists $c>0$ such that
\begin{equation}\label{eq:coerc}
a( T, T)\ \ge\ c\,\lVert T \rVert_{V_1}^2\qquad \forall\, T\in V_1,
\end{equation}
i.e. $a$ is coercive on $V_1$.

\smallskip
\noindent\emph{Existence and uniqueness of a solution to the normal equation.}
By the Lax--Milgram theorem, there exists a unique $ T_{A}^\star\in V_1$ such that
\begin{equation}\label{eq:normal}
a( T_{A}^\star,\eta)\ =\ f(\eta)\qquad \forall\,\eta\in V_1,
\end{equation}
i.e.
\[
\mathcal Q_{A}( T_{A}^\star,\eta)\ =\ \mathcal L_{A}(\eta)\qquad \forall\,\eta\in V_1.
\]
Moreover, taking $\eta= T_{A}^\star$ in \eqref{eq:normal} and using \eqref{eq:coerc} yields
\[
c\,\lVert  T_{A}^\star\rVert_{V_1}^2\ \le\ a( T_{A}^\star, T_{A}^\star)
\ =\ f( T_{A}^\star)
\ \le\lVert f\rVert_{(V_1)^*}\,\lVert  T_{A}^\star\rVert_{V_1},
\]
hence the a priori bound
\begin{equation}\label{eq:apriori}
\lVert T_{A}^\star\rVert_{V_1}\ \le\ c^{-1}\,\lVert f\rVert_{(V_1)^*}
\ =\ c^{-1}\,\lVert \mathcal L_{A}\rVert_{(V_1)^*}.
\end{equation}

\smallskip
\noindent\emph{Optimality for the quadratic functional.}
Write
\[
R_{A}(T)\ =\ C_{A}\ -\ 2\,\mathcal L_{A}(T)\ +\ \mathcal Q_{A}( T, T)
\ =\ C_{A}\ -\ 2\,f(T)\ +\ a( T, T),
\qquad  T\in V_1.
\]
The Fr\'echet derivative of $R_{A}$ at $ T$ in direction $\eta$ is $DR_{A}(T)[\eta] = 2\left(a( T,\eta)-f(\eta)\right).$
Thus $ T$ is a critical point iff it satisfies \eqref{eq:normal}. Since $a$ is
coercive, $R_{A}$ is strictly (indeed, strongly) convex:
\[
R_{A}( T+\eta)\ \ge\ R_{A}(T)\ +\ DR_{A}(T)[\eta]\ +\ c\,\lVert\eta\rVert_{V_1}^2,
\]
by \eqref{eq:coerc}. Therefore the unique solution $ T_{A}^\star$ of \eqref{eq:normal}
is the unique minimizer of $R_{A}$ on $V_1$, and it satisfies the bound \eqref{eq:apriori}. This proves (i).

\noindent\emph{Proof of (ii).}
Fix orthonormal bases $(\phi_k)_{k\ge1}$ of $\mathbb H_o$ and $(\psi_\ell)_{\ell\ge1}$ of $\H$.
For $T\in V_1=\mathcal{HS}(\H^*,\mathbb H_o)$ define its coefficient array
$v=(v_{k,\ell})_{k,\ell}\in\ell^2$ by
\[
v_{k,\ell}:=\big\langle T(\psi_\ell),\phi_k\big\rangle_{\mathbb H_o},\qquad k,\ell\ge1.
\]
Then the Hilbert--Schmidt isometry yields
\begin{equation}\label{eq:HS-isometry}
\lVert T\rVert_{V_1}^2=\sum_{k,\ell\ge1}|v_{k,\ell}|^2=\lVert v\rVert_{\ell^2}^2.
\end{equation}
Moreover, for every $x\in\H^*$ we have the expansion
\begin{equation}\label{eq:T-expansion}
T(x)=\sum_{k,\ell\ge1} v_{k,\ell}\,\langle \psi_\ell,x\rangle_{\H,\H^*}\,\phi_k,
\end{equation}
where the series converges in $\mathbb H_o$.

Define $b\in\ell^2$ and the bounded positive semidefinite operator $\Sigma:\ell^2\to\ell^2$ by
\small
\[
b_{k,\ell}:=\int_\Omega \big\langle Y^{A},\phi_k\big\rangle_{\mathbb H_o}\,
\big\langle \psi_\ell,X^{A}\big\rangle_{\H,\H^*}\,d\mu,
\]
\normalsize
and
\small
\[
\Sigma^{(k,\ell)\,(k',\ell')}:=
\delta_{k,k'}\int_\Omega
\big\langle \psi_\ell,X^{A}\big\rangle_{\H,\H^*}\,
\big\langle \psi_{\ell'},X^{A}\big\rangle_{\H,\H^*}\,d\mu,
\qquad k,k',\ell,\ell'\ge1,
\]
\normalsize
so that $\Sigma$ acts on $\ell^2$ by left multiplication.

We now rewrite the quadratic functional $R_A(T)=\lVert T(X^A)-Y^A\rVert_{\mathcal H}^2$ in coordinates.
First, expanding $T(X^A)$ in the basis $(\phi_k)$ and using \eqref{eq:T-expansion} gives
\small
\[
\big\langle T(X^A),\phi_k\big\rangle_{\mathbb H_o}
=\sum_{\ell\ge1} v_{k,\ell}\,\big\langle \psi_\ell,X^A\big\rangle_{\H,\H^*},
\qquad k\ge1,
\]
\normalsize
with convergence in $L^2(\Omega)$.
Hence, by Parseval in $\mathbb H_o$ and Fubini--Tonelli,
\small
\begin{align*}
\mathcal Q_A(T,T)
&=\int_\Omega \lVert T(X^A)\rVert_{\mathbb H_o}^2\,d\mu
=\sum_{k\ge1}\int_\Omega \big|\langle T(X^A),\phi_k\rangle_{\mathbb H_o}\big|^2\,d\mu \\
&=\sum_{k\ge1}\int_\Omega \Big|\sum_{\ell\ge1} v_{k,\ell}\,\langle \psi_\ell,X^A\rangle_{\H,\H^*}\Big|^2\,d\mu \\
&=\sum_{k\ge1}\sum_{\ell,\ell'\ge1} v_{k,\ell}\,v_{k,\ell'}\,
\int_\Omega \langle \psi_\ell,X^A\rangle_{\H,\H^*}\,
\langle \psi_{\ell'},X^A\rangle_{\H,\H^*}\,d\mu \\
&=\sum_{k,\ell,\ell'\ge1} v_{k,\ell}\,\Sigma^{(k,\ell)\,(k,\ell')}\,v_{k,\ell'}
=\langle v,\Sigma v\rangle_{\ell^2}.
\end{align*}
\normalsize
Similarly,
\small
\begin{align*}
\mathcal L_A(T)
&=\int_\Omega \langle Y^A,T(X^A)\rangle_{\mathbb H_o}\,d\mu
=\sum_{k\ge1}\int_\Omega \langle Y^A,\phi_k\rangle_{\mathbb H_o}\,
\langle T(X^A),\phi_k\rangle_{\mathbb H_o}\,d\mu \\
&=\sum_{k,\ell\ge1} v_{k,\ell}\int_\Omega
\langle Y^A,\phi_k\rangle_{\mathbb H_o}\,
\langle \psi_\ell,X^A\rangle_{\H,\H^*}\,d\mu
=\langle v,b\rangle_{\ell^2}.
\end{align*}
\normalsize
Therefore
\[
R_A(T)=C_A-2\,\langle v,b\rangle_{\ell^2}+\langle v,\Sigma v\rangle_{\ell^2}
=:F(v),
\qquad v\in\ell^2.
\]

Under the coercivity assumption of (i), equivalently $\Sigma\succeq c\,I$ on $\ell^2$,
the functional $F$ is strictly convex and Fr\'echet differentiable on $\ell^2$, with
\[
DF(v)(h)=2\,\langle \Sigma v-b,\,h\rangle_{\ell^2},\qquad v,h\in\ell^2.
\]
Hence $DF(v^\star)=0$ iff $\Sigma v^\star=b$, and by strict convexity this is also sufficient.
Thus the unique minimizer has coefficient vector $v^\star=\Sigma^{-1}b\in\ell^2$, and in particular
$
\sum_{k,\ell\ge1}|v^\star_{k,\ell}|^2<\infty.
$
Finally, inserting $v^\star$ into \eqref{eq:T-expansion} yields the coordinate formula stated in (ii),
and the series converges in $\mathbb H_o$ for each $x\in\H^*$. This proves (ii)

\noindent For the proof of (iii), let $R:V\to V^*$ be the Riesz isomorphism.
The bilinear form $\mathcal Q_A$ induces a bounded self--adjoint positive semidefinite operator
$K:V_1\to V_1^*$ via
\[
(KT)(\eta):=\mathcal Q_A(T,\eta),\qquad T,\eta\in V_1.
\]
Set $B:=R^{-1}K:V_1\to V_1$. Then $B$ is bounded, self--adjoint and positive semidefinite. Indeed,
for $T,\eta\in V_1$,
\[
\langle BT,\eta\rangle_{V_1}
=\langle R^{-1}KT,\eta\rangle_{V_1}
=(KT)(\eta)
=\mathcal Q_A(T,\eta)
=\mathcal Q_A(\eta,T)
=(K\eta)(T)
=\langle B\eta,T\rangle_{V_1},
\]
so $B$ is self--adjoint, and $\langle BT,T\rangle_{V_1}=\mathcal Q_A(T,T)\ge0$. Write $\ell:=R^{-1}\mathcal L_A\in V_1$. Then
$
R_A(T)=C_A-2\,\langle \ell,T\rangle_{V_1}+\langle BT,T\rangle_{V_1},
$
and any critical point satisfies the normal equation
$
BT=\ell\quad\text{in }V_1.
$
Assume $\mathcal L_A\in\overline{\mathrm{Range}(K)}\subset V_1^*$. Applying $R^{-1}$ yields
$\ell\in\overline{\mathrm{Range}(B)}\subset V_1$.
By the spectral theorem, there exists a projection--valued measure $E(\cdot)$ such that
$B=\int_{\sigma(B)}\lambda\,dE(\lambda)$. Define the Moore--Penrose pseudoinverse
$
B^\dagger:=\int_{\sigma(B)\setminus\{0\}}\lambda^{-1}\,dE(\lambda),
$
which is bounded on $\overline{\mathrm{Range}(B)}$. Moreover,
$BB^\dagger=P_{\overline{\mathrm{Range}(B)}}$, hence $BB^\dagger \ell=\ell$.
Therefore $T^\star:=B^\dagger\ell$ solves $BT=\ell$.
The set of all solutions is $T^\star+\ker(B)$, and the unique solution of minimal $V_1$--norm is
$T^\star=B^\dagger\ell$.

\noindent In the coordinate realization from (ii), $B$ corresponds to the positive semidefinite operator
$\Sigma$ on $\ell^2$, $\ell$ corresponds to $b$, and $B^\dagger$ corresponds to $\Sigma^\dagger$.
Hence the minimal norm coefficient vector is $v^\star=\Sigma^\dagger b\in\ell^2$, and
$T^\star=\sum_{k,\ell} v^\star_{k,\ell}\,\phi_k\otimes\psi_\ell$ converges in $V_1$.

\end{proof}

\end{document}